\pdfoutput=1
\RequirePackage{ifpdf}
\ifpdf 
\documentclass[pdftex]{sigma}
\else
\documentclass{sigma}
\fi

\newtheorem{Theorem}{Theorem}[section]

\newtheorem{Lemma}[Theorem]{Lemma}

 { \theoremstyle{definition}

\newtheorem{Remark}[Theorem]{Remark} }

\numberwithin{equation}{section}

\usepackage[mathscr]{eucal}
\newcommand{\ZZ}{\mathbb{Z}}

\begin{document}
\allowdisplaybreaks

\renewcommand{\thefootnote}{}

\newcommand{\arXivNumber}{2410.15408}

\renewcommand{\PaperNumber}{021}

\FirstPageHeading

\ShortArticleName{Bailey Pairs and an Identity of Chern--Li--Stanton--Xue--Yee}

\ArticleName{Bailey Pairs and an Identity\\ of Chern--Li--Stanton--Xue--Yee\footnote{This paper is a~contribution to the Special Issue on Basic Hypergeometric Series Associated with Root Systems and Applications in honor of Stephen C.~Milne's 75th birthday. The~full collection is available at \href{https://www.emis.de/journals/SIGMA/Milne.html}{https://www.emis.de/journals/SIGMA/Milne.html}.
This paper is a~contribution to the Special Issue on Recent Advances in Vertex Operator Algebras in honor of James Lepowsky's 80th birthday. The~full collection is available at \href{https://www.emis.de/journals/SIGMA/Lepowsky.html}{https://www.emis.de/journals/SIGMA/Lepowsky.html}
}}

\Author{Shashank KANADE~$^{\rm a}$ and Jeremy LOVEJOY~$^{\rm b}$}

\AuthorNameForHeading{S.~Kanade and J.~Lovejoy}

\Address{$^{\rm a)}$~Department of Mathematics, University of Denver,\\
\hphantom{$^{\rm a)}$}~2390 South York Street, Denver, Colorado 80210, USA}
\EmailD{\href{mailto:shashank.kanade@du.edu}{shashank.kanade@du.edu}}
\URLaddressD{\url{https://cs.du.edu/~shakanad/}}

\Address{$^{\rm b)}$~CNRS, Universit{\'e} Paris Cit\'e, B\^atiment Sophie Germain,\\
\hphantom{$^{\rm b)}$}~Case Courier 7014, 8 Place Aur\'elie Nemours, 75205 Paris Cedex 13, France}
\EmailD{\href{mailto:lovejoy@math.cnrs.fr}{lovejoy@math.cnrs.fr}}
\URLaddressD{\url{https://lovejoy.perso.math.cnrs.fr}}

\ArticleDates{Received October 28, 2024, in final form March 19, 2025; Published online March 29, 2025}

\Abstract{We show how Bailey pairs can be used to give a simple proof of an identity of Chern, Li, Stanton, Xue, and Yee. The same method yields a number of related identities as well as false theta companions.}


\Classification{33D15}

\begin{flushright}
\begin{minipage}{65mm}
\it Dedicated to James Lepowsky\\ on the occasion of his 80th birthday and\\ Stephen Milne on the occasion\\ of his 75th birthday
\end{minipage}
\end{flushright}

\renewcommand{\thefootnote}{\arabic{footnote}}
\setcounter{footnote}{0}

\section{Introduction}
Recall the usual $q$-series notation,
\begin{equation*}
	(a_1,a_2,\dots ,a_k)_{\infty} = (a_1,a_2,\dots ,a_k;q)_{\infty} = \prod_{j=0}^{\infty} \bigl(1-a_1q^j\bigr)\bigl(1-a_2q^j\bigr) \cdots \bigl(1-a_kq^j\bigr)
\end{equation*}
and
\begin{equation*}
	(a_1,a_2,\dots ,a_k)_n = (a_1,a_2,\dots ,a_k;q)_n = \prod_{j=0}^{n-1} \bigl(1-a_1q^j\bigr)\bigl(1-a_2q^j\bigr) \cdots \bigl(1-a_kq^j\bigr),
\end{equation*}
valid for $n \geq 0$, along with the $q$-binomial coefficient
\begin{equation} \label{eqn:qbinom}
	\begin{bmatrix} n \\ k \end{bmatrix} = \begin{bmatrix} n \\ k \end{bmatrix}_q =
	\begin{cases}
		\displaystyle\frac{(q)_n}{(q)_{n-k}(q)_k} & \text{if $0 \leq k \leq n$}, \\
		0 & \text{otherwise}.
	\end{cases}
\end{equation}
In a recent study of $q$-series and partitions related to Ariki--Koike algebras, Chern, Li, Stanton, Xue, and Yee \cite{CLSXY} established the following family of $q$-multisum identities.

\begin{Theorem} \label{CLSXYthm}
	Let $m \geq 1$ and $0 \leq a \leq m-1$. Then we have
	\begin{equation} \label{CLSXYeq}
		\sum_{n_m, \dots, n_1 \geq 0} \frac{q^{\binom{n_m + 1}{2} + \cdots + \binom{n_1+1}{2}}}{(q)_{n_m}} \prod_{i=1}^{m-1} \begin{bmatrix} n_{i+1} + \delta_{a,i} \\ n_i \end{bmatrix} = \frac{\bigl(q^{a+1},q^{m+1-a},q^{m+2};q^{m+2}\bigr)_{\infty}}{(q)_{\infty}\bigl(q;q^2\bigr)_{\infty}}.
	\end{equation}
\end{Theorem}
\noindent This generalizes a classical identity in the theory of partitions \cite[equation~(2.26), $t=q$]{AndBook},
\begin{equation*}
	\sum_{n \geq 0} \frac{q^{\binom{n+1}{2}}}{(q)_n} = \frac{1}{\bigl(q;q^2\bigr)_{\infty}} = (-q)_{\infty}.
\end{equation*}
The proof of \eqref{CLSXYeq} in \cite{CLSXY} is lengthy and impressive, involving a symmetry property, a $q$-binomial coefficient multisum transformation formula, and two identities of Andrews \cite{AndParity} and Kim--Yee~\cite{KimYee}.

In the first part of this paper, we give a streamlined proof of \eqref{CLSXYeq} using the Bailey pair machinery. All that we require are the classical Bailey lemma and Bailey lattice along with a~Bailey-type lemma from \cite{Lo1} -- see Lemma \ref{littlelemma}. In fact, we establish a much more general result, which allows us to prove \eqref{CLSXYeq} and many more families of identities like it.
\begin{Theorem}
	\label{thm:main}
	If $(\alpha_n,\beta_n)$ is a Bailey pair relative to $q$ and $f_n$ is a sequence defined for all integers $n$ which satisfies
	\begin{align} \label{eqn:fdef}
		\alpha_n=\dfrac{1-q^{2n+1}}{1-q}f_n,\qquad n\geq 0
	\end{align}
	and
	\begin{align}
		q^{n}f_n = -q^{-n-1}f_{-n-1}, \qquad n\in \ZZ,
		\label{eqn:fcond}
	\end{align}
	then for all $m \geq 1$ and $0 \leq a \leq m$, we have
	\begin{gather}
			\sum_{n_m, \dots, n_1 \geq 0} \frac{q^{\binom{n_m+1}{2} + \cdots + \binom{n_1+1}{2}}\bigl(q^2;q^2\bigr)_{n_1+\delta_{a,0}}}{(q)_{n_m}} \beta_{n_1+\delta_{a,0}} \prod_{i=1}^{m-1} \begin{bmatrix} n_{i+1} + \delta_{a,i} \\ n_i \end{bmatrix}\nonumber \\
			\qquad{}= \dfrac{(-q)_\infty}{(q)_\infty}
			\sum_{n\in\ZZ}
			q^{m\binom{n}{2}+an}f_n.\label{eqn:mainid}
	\end{gather}
\end{Theorem}
Note that using \eqref{eqn:fcond}, the condition \eqref{eqn:fdef} could be written alternatively as:
\begin{align*}
	\alpha_n=\dfrac{f_n+f_{-n-1}}{1-q},\qquad n\geq 0.
\end{align*}
Theorem \ref{CLSXYthm} now follows using the classical Bailey pair relative to $q$ \cite[E(3)]{Sl1},
\begin{equation} \label{seed1alpha}
	\alpha_n = \frac{1-q^{2n+1}}{1-q}(-1)^nq^{n^2}
\end{equation}
and
\begin{equation} \label{seed1beta}
	\beta_n = \frac{1}{\bigl(q^2;q^2\bigr)_n},
\end{equation}
together with the triple product identity \eqref{eqn:jtp}. Other Bailey pairs lead to similar families of identities, and we give a number of examples in Section \ref{sec:mainthm} -- see Theorems \ref{q^2thm}, \ref{strangesignthm}, and \ref{thm:quintuple}.

In the second part of the paper, we prove a result similar to Theorem \ref{thm:main} but involving false theta functions instead of theta functions. To state it, we define the function
\begin{equation*}
	{\rm sgn}(n) =
	\begin{cases}
		\hphantom{-}1 &\text{if $n \geq 0$}, \\
		-1 &\text{if $n < 0$}.
	\end{cases}
\end{equation*}

\begin{Theorem}
	\label{thm:false}
	If $(\alpha_n,\beta_n)$ is a Bailey pair relative to $q$ and $f_n$ is as in \eqref{eqn:fdef} and \eqref{eqn:fcond}, then for all $m\geq 1$ and $0\leq a \leq m$, we have
	\begin{gather}
			\sum_{n_m, \dots, n_1 \geq 0} \frac{(-1)^{n_m}q^{\binom{n_m+1}{2} + \cdots + \binom{n_1+1}{2}}\bigl(q^2;q^2\bigr)_{n_1+\delta_{a,0}}}{(-q)_{n_m}} \beta_{n_1+\delta_{a,0}} \prod_{i=1}^{m-1} \begin{bmatrix} n_{i+1} + \delta_{a,i} \\ n_i \end{bmatrix}
			\nonumber\\
			\qquad{}=
			\begin{cases}
				\displaystyle\sum_{n\in\ZZ}
				{\rm sgn}(-n)(-1)^nq^{m\binom{n}{2}+an}f_n & \text{if $0 \leq a \leq m-1$}, \vspace{1mm}\\
				\displaystyle\sum_{n \in \mathbb{Z}} {\rm sgn}(n)(-1)^n q^{m\binom{n+1}{2}}f_n & \text{if $a=m$}.
			\end{cases}\label{eqn:falseid}
	\end{gather}
\end{Theorem}
Each family of identities arising from Theorem \ref{thm:main} then has a false theta counterpart using Theorem \ref{thm:false}. In the case of \eqref{CLSXYeq}, using \eqref{seed1alpha} and \eqref{seed1beta} this is the following.
\begin{Theorem}\label{thm:falseCLSXY}
	For $m \geq 1$ and $0 \leq a \leq m-1$, we have
	\begin{gather*}
			\sum_{n_m, \dots, n_1 \geq 0} \frac{(-1)^{n_m}q^{\binom{n_m + 1}{2} + \cdots + \binom{n_1+1}{2}}}{(-q)_{n_m}} \prod_{i=1}^{m-1} \begin{bmatrix} n_{i+1} + \delta_{a,i} \\ n_i \end{bmatrix}
		 = \sum_{n \in \mathbb{Z}} {\rm sgn}(-n)q^{(m+2)\binom{n}{2} + (a+1)n}.
 \end{gather*}
\end{Theorem}
The base case $m=1$ and $a=0$ is the well-known
\begin{equation*}
	\sum_{n \geq 0} \frac{(-1)^nq^{\binom{n+1}{2}}}{(-q)_n} = \sum_{n \geq 0} q^{n(3n+1)/2}\bigl(1-q^{2n+1}\bigr).
\end{equation*}	
See Theorems \ref{thm:falseq^2}--\ref{thm:falsequintuple} for further examples.

In the final part of the paper we are motivated by the series $\mathscr{S}_{m,a}$, defined for $m \geq 1$ and~${0 \leq a \leq m}$ by
\begin{align*}
	\mathscr{S}_{m,a} = \sum_{n_m, \dots, n_1 \geq 0} \frac{q^{n_m^2 + \cdots + n_1^2}}{\bigl(q^2;q^2\bigr)_{n_m}} \prod_{i=1}^{m-1} \begin{bmatrix} n_{i+1} + \delta_{a,i} \\ n_i \end{bmatrix}_{q^2}.
\end{align*}
These are naturally dilated versions of the series in \eqref{CLSXYeq}, and it is known that we have
\begin{align} \label{eqn:Smm}
	\mathscr{S}_{m,0}=\mathscr{S}_{m,m} = \dfrac{\bigl(-q;q^2\bigr)_\infty}{\bigl(q^2;q^2\bigr)_\infty}\bigl(q^{m+1},q^{m+3},q^{2m+4}; q^{2m+4}\bigr)_\infty.
\end{align}
See \cite[Corollary 1.5\,(b)]{St} for general $m$ and \cite[Section 5]{BIS} for $m$ even. While it appears that the~$\mathscr{S}_{m,a}$ are not infinite products for $a\not \in \{0,m\}$, we prove the following result on the differences of these series.

\begin{Theorem}
	\label{thm:dilationex1}
	For $1\leq a\leq m$, we have
	\begin{align*}
		\mathscr{S}_{m,a}-\mathscr{S}_{m,a-1}=q^{a}\dfrac{\bigl(-q;q^2\bigr)_\infty}{\bigl(q^2;q^2\bigr)_\infty}\bigl(q^{m+1-2a},q^{m+3+2a},q^{2m+4};q^{2m+4}\bigr)_\infty.
	\end{align*}
\end{Theorem}
This and several similar families of identities will follow as special cases of our third main theorem.
\begin{Theorem} \label{thm:dilation}
	Suppose that $(\alpha_n,\beta_n)$ is a Bailey pair relative to $1$ and that there is a sequence~$g_n$, defined for all integers $n$, with $g_0 = 1$,
	\begin{equation}
		\alpha_n =
		\begin{cases}
			1 &\text{if $n= 0$}, \\
			(1+q^n)g_n &\text{if $n > 1$},
		\end{cases}
		\label{eqn:gdef}
	\end{equation}
	and
	\begin{equation} \label{eqn:gcond}
		g_{-n} = q^ng_n.
	\end{equation}
	Then for all $m \geq 0$ and $1 \leq a \leq m$, we have
	\begin{gather}
			\sum_{n_{m},\dots,n_1\geq 0}
			\frac{q^{\frac{1}{2}(n_{m}^2+n_{m-1}^2+\cdots+n_1^2)}\bigl(-q^{\frac{1}{2}},q\bigr)_{n_1}}{(q)_{n_m}}\beta_{n_1}
			\prod_{i=1}^{m-1} \begin{bmatrix} n_{i+1} + \delta_{a,i} \\ n_i \end{bmatrix} \nonumber\\
			\qquad\quad{} - \sum_{n_{m},\dots,n_1\geq 0}
			\frac{q^{\frac{1}{2}(n_{m}^2+n_{m-1}^2+\cdots+n_1^2)}\bigl(-q^{\frac{1}{2}},q\bigr)_{n_1 + \delta_{a-1,0}}}{(q)_{n_m}}\beta_{n_1 + \delta_{a-1,0}}
			\prod_{i=1}^{m-1} \begin{bmatrix} n_{i+1} + \delta_{a-1,i} \\ n_i \end{bmatrix} \nonumber\\
			\qquad{} = -q^{\frac{a-1}{2}} \frac{\bigl(-q^{\frac{1}{2}}\bigr)_{\infty}}{(q)_{\infty}} \sum_{n \in \mathbb{Z}} q^{\frac{m}{2}n^2+na-n}g_{n+1}.\label{eqn:dilationid}
	\end{gather}
\end{Theorem}	
Note that with \eqref{eqn:gcond}, the condition \eqref{eqn:gdef} could be written alternatively as
\begin{equation*}
	\alpha_n =
	\begin{cases}
		1 &\text{if $n= 0$}, \\
		g_n+g_{-n} &\text{if $n > 1$}.
	\end{cases}
\end{equation*}

The paper is organized as follows. In the next section, we collect some basic facts about Bailey pairs. In Section \ref{sec:mainthm}, we prove Theorem \ref{thm:main} and give several applications, including Theorem \ref{CLSXYthm}. In Section \ref{sec:falsetheta}, we prove Theorem \ref{thm:false} and give false theta companions for each family of identities in Section \ref{sec:mainthm}. In Section \ref{sec:dilated}, we prove Theorem \ref{thm:dilation}, which allows us to deduce Theorem \ref{thm:dilationex1} and other similar identities.
Wherever possible, we include historical remarks about the ``base cases'' of our identities.

While it may be striking that identities involving products of $q$-binomial coefficients
\begin{equation}
	\prod_{i=1}^{m-1} \begin{bmatrix} n_{i+1} + \delta_{a,i} \\ n_i \end{bmatrix}
	\label{eqn:shiftedbinom}
\end{equation}
are so widespread, the fact that the Bailey pair machinery can be used to prove such identities should not be a surprise. This is the most powerful and systematic technique for treating $q$-multisum identities, and a number of identities with products of $q$-binomial coefficients like~\eqref{eqn:shiftedbinom} have already appeared in the literature in connection to Bailey pairs. We discuss these briefly in the concluding remarks at the end of the paper.

\section{Bailey pairs}
In this section, we review the necessary background on Bailey pairs. A Bailey pair relative to~$x$~\cite{Andmultiple} is a pair of sequences $(\alpha_n,\beta_n)$,
$n\in \ZZ_{\geq 0}$,
satisfying
\begin{equation*} 
	\beta_n = \sum_{k=0}^n \frac{\alpha_k}{(q)_{n-k}(xq)_{n+k}}.
\end{equation*}
Note that in the limit, this gives (subject to convergence conditions)
\begin{equation} \label{Baileylimit}
	\lim_{n \to \infty} \beta_n = \frac{1}{(q,xq)_{\infty}} \sum_{k=0}^{\infty} \alpha_k.
\end{equation}
In practice, the above sum often becomes an infinite product by an appeal to the Jacobi triple product identity \cite[equation~(2.2.10)]{AndBook},
\begin{align}
	\sum_{n\in\ZZ}(-1)^n z^nq^{n^2}=\bigl(q^2,zq,z^{-1}q;q^2\bigr)_\infty,
	\label{eqn:jtp}
\end{align}
or the quintuple product identity \cite{C-quint},
\begin{align}
	\Biggl(\sum_{\substack{n\in\ZZ\\ n\equiv 0\,\,(\mathrm{mod}\,\,3)}}
	-\sum_{\substack{n\in\ZZ\\ n\equiv 2\,\,(\mathrm{mod}\,\,3)}}
	\Biggr) z^nq^{\frac{1}{3}\binom{n+1}{2}}=\bigl(q,zq,z^{-1}\bigr)_\infty\bigl(z^2q,z^{-2}q;q^2\bigr)_\infty
	\label{eqn:qtp}.
\end{align}

The most important fact about Bailey pairs is the following, which is known as the Bailey lemma. From a given Bailey pair relative to $x$, it produces new Bailey pairs relative to $x$.

\begin{Lemma}[\cite{Andmultiple}] \label{Baileylemma}
	If $(\alpha_n,\beta_n)$ is a Bailey pair relative to $x$, then so is $(\alpha_n',\beta_n')$, where
	\begin{equation*} 
		\alpha_n' = \frac{(b,c)_n(xq/bc)^n}{(xq/b,xq/c)_n}\alpha_n
	\qquad
	\text{and}
	\qquad 
		\beta_n' = \frac{1}{(xq/b,xq/c)_n} \sum_{j=0}^n \frac{(b,c)_j(xq/bc)_{n-j} (xq/bc)^j}{(q)_{n-j}} \beta_j.
	\end{equation*}
\end{Lemma}

A result similar to Lemma \ref{Baileylemma} takes a Bailey pair relative to $x$ and produces Bailey pairs relative to $x/q$.
\begin{Lemma}[\cite{AABLattice}] \label{Baileylatticelemma}
	If $(\alpha_n,\beta_n)$ is a Bailey pair relative to $x$, then $(\alpha_n',\beta_n')$ is a Bailey pair relative to $x/q$, where
	\begin{equation*} 
		\alpha_n' = (1-x) \Bigl(\frac{x}{bc} \Bigr) \frac{(b,c)_n(x/bc)^n}{(x/b,x/c)_n}\biggl(\frac{\alpha_n}{1-xq^{2n}} - \frac{xq^{2n-2}\alpha_{n-1}}{1-xq^{2n-2}} \biggr)
	\end{equation*}
and
	\begin{equation*} 
		\beta_n' = \frac{1}{(x/b,x/c)_n} \sum_{j=0}^n \frac{(b,c)_j(x/bc)_{n-j} (x/bc)^j}{(q)_{n-j}} \beta_j.
	\end{equation*}
Here by convention we take
	\begin{equation} \label{eqn:latticeextracond}
		\alpha_{-1} = 0.
	\end{equation}
\end{Lemma}

Finally, we have two lemmas that change a Bailey pair relative to $x$ to one relative to $xq$. The first will be key to introducing $q$-binomial coefficients of the form $\bigl[\begin{smallmatrix} n_{i+1} + 1 \\ n_i \end{smallmatrix}\bigr]$. The second is not needed for the proof of Theorem \ref{CLSXYthm} but will be used to establish a certain Bailey pair later in the paper.
\begin{Lemma}[\cite{Lo1}] \label{littlelemma}
	If $(\alpha_n,\beta_n)$ is a Bailey pair relative to $x$, then $(\alpha_n',\beta_n')$ is a Bailey pair relative to $xq$, where
	\begin{equation*} 
		\alpha_n' = \frac{1}{1-xq} \biggl(\frac{1-q^{n+1}}{1-xq^{2n+2}} \alpha_{n+1} + \frac{q^n(1-xq^n)}{1-xq^{2n}} \alpha_n \biggr)
	\qquad
	\text{and}
	\qquad 
		\beta_n' = \bigl(1-q^{n+1}\bigr)\beta_{n+1}.
	\end{equation*}
\end{Lemma}

\begin{Lemma}[\cite{Lo0}] \label{1toqlemma}
	If $(\alpha_n,\beta_n)$ is a Bailey pair relative to $x$, then $(\alpha_n',\beta_n')$ is a Bailey pair relative to $xq$, where
	\begin{equation*}
		\alpha_n' =\frac{\bigl(1-xq^{2n+1}\bigr)(xq/b)_n(-b)^nq^{\binom{n}{2}}}{(1-xq)(bq)_n} \sum_{j=0}^n \frac{(b)_j}{(xq/b)_j}(-b)^{-j}q^{-\binom{j}{2}}\alpha_j
	\qquad
	\text{and}
	\qquad
		\beta_n' = \frac{(b)_n}{(bq)_n} \beta_n.
	\end{equation*}
\end{Lemma}

\section{Proof of Theorem \ref{thm:main} and applications}
\label{sec:mainthm}

We begin this section by establishing a key Bailey lemma for the proof of Theorem \ref{thm:main}.
\begin{Lemma} \label{lemma:mainpair}
	Let $m \geq 1$ and $0 \leq a \leq m$. Suppose that $(\alpha_n,\beta_n)$ is a Bailey pair relative to $q$ and that there is a sequence $(f_n)_{n \geq -1}$ such that
	\begin{align} \label{eqn:fdeflater}
		\alpha_n=\dfrac{1-q^{2n+1}}{1-q}f_n,
	\end{align}
	where
	\begin{align}
		q^{-1}f_{-1}=-f_0.
		\label{eqn:f-1f0}
	\end{align}
	Then $(\alpha_n',\beta_n')$ is a Bailey pair relative to $q$, where
	\begin{equation*} 
		\alpha_n' =
		\begin{cases}
			\dfrac{1}{1-q}q^{m\binom{n+1}{2}}\bigl(
			q^{a(n+1)}f_{n+1} -q^{-an+2n-1}f_{n-1}
			\bigr) & \text{if $0 \leq a \leq m-1$}, \vspace{1mm}\\
			\dfrac{1}{1-q}q^{m \binom{n+1}{2}}\bigl(1-q^{2n+1}\bigr)f_n &\text{if $a=m$},
		\end{cases}
	\end{equation*}
	and
	\begin{align}
		\beta_n' &{}= \beta_{n_{m+1}}' \nonumber \\
 &{}= \frac{1}{(-q)_{n_{m+1}}} \sum_{n_m, \dots, n_1 \geq 0} \frac{q^{\binom{n_m+1}{2} + \cdots + \binom{n_1+1}{2}}\bigl(q^2;q^2\bigr)_{n_1+\delta_{a,0}}}{(q)_{n_{m+1}-n_m} (q)_{n_m}} \beta_{n_1+\delta_{a,0}} \prod_{i=1}^{m-1} \begin{bmatrix} n_{i+1} + \delta_{a,i} \\ n_i \end{bmatrix}.\label{eqn:mainpairbeta}
	\end{align}
	
\end{Lemma}

\begin{proof}
	{\it Step 1\emph{:}} First assume $a > 0$. Iterate Lemma \ref{Baileylemma} $a$ times with $x\mapsto q$, $b\mapsto -q$, and $c\rightarrow \infty$ to obtain another Bailey pair relative to $q$. The resulting $\alpha_n$ and $\beta_n$ are
	\begin{align}
		\alpha_n=q^{a\binom{n+1}{2}}\dfrac{1-q^{2n+1}}{1-q}f_n
		\label{eqn:alphafn}
	\end{align}
	and
	\begin{align}
		\beta_n=\beta_{n_{a+1}}=
		\dfrac{1}{(-q)_{n_{a+1}}}\sum_{n_a,\dots,n_1\geq 0}
		\dfrac{q^{\binom{n_a+1}{2}+\cdots+\binom{n_1+1}{2}}(-q)_{n_1}}{(q)_{n_{a+1}-n_a}\cdots (q)_{n_2-n_1}}\beta_{n_1}. \label{eqn:betastep1}
	\end{align}
	Here and throughout the paper we use the convention that $1/(q)_n = 0$ if $n <0$.
	If $a=m$, we stop here. Multiplying the numerator and denominator of the above sum by $(q)_{n_1} \cdots (q)_{n_m}$ and rewriting in terms of $q$-binomial coefficients using \eqref{eqn:qbinom} gives \eqref{eqn:mainpairbeta} for $a=m$.
	
	{\it Step 2\emph{:}} Apply Lemma \ref{littlelemma} with $x \mapsto q$ to obtain a Bailey pair relative to $q^2$:
	\begin{align*}
		\alpha_n &{}= \dfrac{1-q^{n+1}}{(q)_2}\bigl(
		q^{a\binom{n+2}{2}}f_{n+1} + q^{n+a\binom{n+1}{2}}f_n\bigr)\\
		&{}=\dfrac{1-q^{n+1}}{(q)_2}q^{(a+2)\binom{n+1}{2}-n^2}
		\bigl(
		f_n+q^{a(n+1)-n}f_{n+1}\bigr)
	\end{align*}
	and
	\begin{align*}
		\beta_n=\beta_{n_{a+1}}=
		\dfrac{1-q^{n_{a+1}+1}}{(-q)_{n_{a+1}+1}}\sum_{n_a,\dots,n_1\geq 0}
		\dfrac{q^{\binom{n_a+1}{2}+\cdots+\binom{n_1+1}{2}}(-q)_{n_1}}{(q)_{n_{a+1}+1-n_a}\cdots (q)_{n_2-n_1}}\beta_{n_1}.
	\end{align*}

	{\it Step 3\emph{:}} Use Lemma \ref{Baileylemma} with $x\mapsto q^2$, $b\mapsto-q^2$, and $c \to \infty$ to obtain another Bailey pair relative to $q^2$. The resulting pair is
	\begin{align*}
		\alpha_n
		=\dfrac{1-q^{2n+2}}{(1+q)(q)_2}q^{(a+3)\binom{n+1}{2}-n^2}
		\bigl(
		f_n+q^{a(n+1)-n}f_{n+1}\bigr)
	\end{align*}
	and
	\begin{align*}
		\beta_n=\beta_{n_{a+2}}=
		\dfrac{1}{(1+q)(-q)_{n_{a+2}}}\sum_{n_{a+1},\dots,n_1\geq 0}
		\dfrac{q^{\binom{n_{a+1}+1}{2}+\cdots+\binom{n_1+1}{2}}\bigl(1-q^{n_{a+1}+1}\bigr)(-q)_{n_1}}{(q)_{n_{a+2}-n_{a+1}}(q)_{n_{a+1}+1-n_a}\cdots (q)_{n_2-n_1}}\beta_{n_1}.
	\end{align*}
	At this point, we multiply both
	$\alpha_n$ and $\beta_n$ by $1+q$.

	{\it Step 4\emph{:}} Use Lemma \ref{Baileylatticelemma} with $x \mapsto q^2$ and $b,c \mapsto q$ to obtain a Bailey pair relative to $q$. This retains the $\beta_n$, but does change the $\alpha_n$. When $n \geq 1$, we have
	\begin{gather*}
		\alpha_n = \dfrac{1}{1-q}
		\bigl(
		q^{(a+3)\binom{n+1}{2}-n^2}\bigl(f_n+q^{a(n+1)-n}f_{n+1}\bigr) \\ \hphantom{\alpha_n = \dfrac{1}{1-q}\bigl(}{}
		- q^{2n+(a+3)\binom{n}{2}-(n-1)^2}\bigl(f_{n-1}+q^{an-n+1}f_{n}\bigr)
		\bigr)\\ \hphantom{\alpha_n}{}
		=
		\dfrac{1}{1-q}
		\bigl(
		q^{(a+3)\binom{n+1}{2}-n^2+a(n+1)-n}f_{n+1}
		-
		q^{2n+(a+3)\binom{n}{2}-(n-1)^2}f_{n-1}
		\bigr)\\ \hphantom{\alpha_n}{}
		= \dfrac{1}{1-q}q^{(a+1)\binom{n+1}{2}}\bigl(
		q^{a(n+1)}f_{n+1} -q^{-an+2n-1}f_{n-1}
		\bigr).
	\end{gather*}
	When $n=0$, the above is also valid, using \eqref{eqn:latticeextracond} and \eqref{eqn:f-1f0}.

	{\it Step 5\emph{:}} Iterate Lemma \ref{Baileylemma} $m-a-1$ times with $x\mapsto q$, $b\mapsto -q$, and $c\rightarrow \infty$ to arrive at the final Bailey pair relative to $q$:
	\begin{align*}
		\alpha_n
		=
		\dfrac{1}{1-q}q^{m\binom{n+1}{2}}\bigl(
		q^{a(n+1)}f_{n+1} -q^{-an+2n-1}f_{n-1}
		\bigr)
	\end{align*}
	and
	\begin{align*}
		\beta_n=\beta_{n_{m+1}}=
		\dfrac{1}{(-q)_{n_{m+1}}}\sum_{n_m,\dots,n_1\geq 0}
		\dfrac{q^{\binom{n_{m}+1}{2}+\cdots+\binom{n_1+1}{2}}\bigl(1-q^{n_{a+1}+1}\bigr)(-q)_{n_1}}
		{(q)_{n_{m+1}-n_{m}}\cdots
			(q)_{n_{a+1}+1-n_a}\cdots (q)_{n_2-n_1}}\beta_{n_1}.
	\end{align*}
	Rewriting the sum in terms of $q$-binomial coefficients gives{\samepage
	\begin{equation*} 
		\beta_n = \beta_{n_{m+1}} = \frac{1}{(-q)_{n_{m+1}}} \sum_{n_m, \dots, n_1 \geq 0} \frac{q^{\binom{n_m+1}{2} + \cdots + \binom{n_1+1}{2}}\bigl(q^2;q^2\bigr)_{n_1}}{(q)_{n_{m+1}-n_m} (q)_{n_m}} \beta_{n_1} \prod_{i=1}^{m-1} \begin{bmatrix} n_{i+1} + \delta_{a,i} \\ n_i \end{bmatrix},
	\end{equation*}
	which coincides with \eqref{eqn:mainpairbeta}.}
	
	{\it Step 6\emph{:}}
	If $a=0$, performing Steps 2--5 leads to the following $\beta_n$ (while the formula for $\alpha_n$ is the same as in Step 5 with $a\mapsto 0$):
	\begin{align*}
		\beta_n=\beta_{n_{m+1}}=
		\dfrac{1}{(-q)_{n_{m+1}}}\sum_{n_m,\dots,n_1\geq 0}
		\dfrac{q^{\binom{n_{m}+1}{2}+\cdots+\binom{n_1+1}{2}}\bigl(1-q^{n_{1}+1}\bigr)(-q)_{n_1+1}}
		{(q)_{n_{m+1}-n_{m}}\cdots (q)_{n_2-n_1}}
		\beta_{n_1+1}.
	\end{align*}
	Rewriting this in terms of $q$-binomial coefficients gives
	\begin{equation*} 
		\beta_n = \beta_{n_{m+1}} = \frac{1}{(-q)_{n_{m+1}}} \sum_{n_m, \dots, n_1 \geq 0} \frac{q^{\binom{n_m+1}{2} + \cdots + \binom{n_1+1}{2}}\bigl(q^2;q^2\bigr)_{n_1+1}}{(q)_{n_{m+1}-n_m} (q)_{n_m}} \beta_{n_1+1} \prod_{i=1}^{m-1} \begin{bmatrix} n_{i+1} + \delta_{a,i} \\ n_i \end{bmatrix},
	\end{equation*}
	which again matches \eqref{eqn:mainpairbeta}. This completes the proof.
\end{proof}

Armed with this Bailey pair, we are now ready to prove Theorem \ref{thm:main}.

\begin{proof}[Proof of Theorem \ref{thm:main}]
	Let $(\alpha_n,\beta_n)$ be a Bailey pair relative to $q$ satisfying \eqref{eqn:fdeflater} and \eqref{eqn:f-1f0}.
	Applying \eqref{Baileylimit} to the Bailey pair in Lemma \ref{lemma:mainpair}, we obtain
	\begin{gather}
			\sum_{n_m, \dots, n_1 \geq 0} \frac{q^{\binom{n_m+1}{2} + \cdots + \binom{n_1+1}{2}}\bigl(q^2;q^2\bigr)_{n_1+\delta_{a,0}}}{(q)_{n_m}} \beta_{n_1+\delta_{a,0}} \prod_{i=1}^{m-1} \begin{bmatrix} n_{i+1} + \delta_{a,i} \\ n_i \end{bmatrix}
			\nonumber\\
			\qquad{}=
			\frac{(-q)_{\infty}}{(q)_\infty}\cdot
			\begin{cases}
				\displaystyle\sum_{n\geq 0}
				q^{m\binom{n+1}{2}}\bigl(
				q^{a(n+1)}f_{n+1} -q^{-an+2n-1}f_{n-1}\bigr)
				& \text{if $0\leq a \leq m-1$},\vspace{1mm}\\
				\displaystyle\sum_{n\geq 0}
				q^{m\binom{n+1}{2}}\bigl(
				1-q^{2n+1}\bigr)f_n
				& \text{if $a=m$}.
			\end{cases} \label{eqn:mainintermediate}
	\end{gather}
	If $f_n$ also satisfies \eqref{eqn:fcond}, then we obtain the right-hand side of \eqref{eqn:mainid} from the right-hand side of~\eqref{eqn:mainintermediate} when $a <m$ by computing as follows:
	\begin{align*}
		&\sum_{n\geq 0}
		q^{m\binom{n+1}{2}}\bigl(
		q^{a(n+1)}f_{n+1} -q^{-an+2n-1}f_{n-1}\bigr) \\
		&\qquad{}=\sum_{n\geq 1}q^{m\binom{n}{2}+an}f_n
		-\sum_{n\geq 0}
		q^{m\binom{n+1}{2}-an+2n-1}f_{n-1}\\
		&\qquad{}=\sum_{n\geq 1}q^{m\binom{n}{2}+an}f_n
		-\sum_{n\leq 0}
		q^{m\binom{n}{2}+an-2n-1}f_{-n-1} \\
		&\qquad{}=\sum_{n\geq 1}q^{m\binom{n}{2}+an}f_n
		+\sum_{n\leq 0}
		q^{m\binom{n}{2}+an}f_{n}.
	\end{align*}
	The case $a=m$ is similar.
\end{proof}

We now give several applications of Theorem \ref{thm:main}, beginning with Theorem \ref{CLSXYthm}.

\begin{proof}[Proof of Theorem \ref{CLSXYthm}]
	Using the Bailey pair in \eqref{seed1alpha} and \eqref{seed1beta}, we have
	\begin{equation*}
		f_n = (-1)^nq^{n^2},
	\end{equation*}
	which is readily seen to satisfy \eqref{eqn:fcond}.
	Applying Theorem \ref{thm:main}, the left-hand side of \eqref{eqn:mainid} becomes the left-hand side of \eqref{CLSXYeq}, while the sum on the right-hand side is
	\begin{align*}
		\sum_{n\in\ZZ}q^{m\binom{n}{2}+an}f_n =
		\sum_{n\in\ZZ}(-1)^nq^{\frac{m+2}{2}n^2+(a-\frac{m}{2})n }
		 =\bigl( q^{a+1} , q^{m+1-a} , q^{m+2};q^{m+2}\bigr)_\infty,
	\end{align*}
	by \eqref{eqn:jtp}.
\end{proof}

Our next two applications of Theorem \ref{thm:main} are contained in the following theorems.

\begin{Theorem} \label{q^2thm}
	For $m \geq 1$ and $0 \leq a \leq {m}$, we have
	\begin{gather}
			\sum_{n_m,\dots ,n_1 \geq 0} \frac{q^{n_m^2+n_m + \cdots + n_1^2+n_1}}{\bigl(q^2;q^2\bigr)_{n_m} \bigl(-q;q^2\bigr)_{n_1 + \delta_{a,0}}} \prod_{i=1}^{m-1} \begin{bmatrix} n_{i+1} + \delta_{a,i} \\ n_i \end{bmatrix}_{q^2} \nonumber\\
			\qquad{} = \frac{\bigl(-q^2;q^2\bigr)_{\infty}\bigl(q^{2a+1},q^{2m-2a+2},q^{2m+3};q^{2m+3}\bigr)_{\infty}}{\bigl(q^2;q^2\bigr)_{\infty}}. \label{q^2eq}
	\end{gather}
\end{Theorem}

\begin{Theorem} \label{strangesignthm}
	For $m \geq 1$ and $0 \leq a \leq {m}$, we have
	\begin{gather}
			\sum_{n_m, \dots, n_1 \geq 0} \frac{q^{\binom{n_m+1}{2} + \cdots + \binom{n_1+1}{2}}\bigl(-1;q^2\bigr)_{n_1 + \delta_{a,0}}}{(q)_{n_m}\bigl(q;q^2\bigr)_{n_1+ \delta_{a,0}}} \prod_{i=1}^{m-1} \begin{bmatrix} n_{i+1} + \delta_{a,i} \\ n_i \end{bmatrix} \nonumber\\
			\qquad{}= \frac{(-q)_{\infty}}{(q)_{\infty}} \bigl(-q^a,q^{m+1+a},q^{m+1-a},-q^{2m+2-a},-q^{m+1},q^{2m+2};q^{2m+2}\bigr)_{\infty}.
\label{strangesignex}
	\end{gather}
\end{Theorem}
For $m = 1$ and $a=0,1$, Theorem \ref{q^2thm} gives the identities
\begin{equation*}
	\sum_{n \geq 0} \frac{q^{n^2+n}}{\bigl(q^2;q^2\bigr)_n\bigl(-q;q^2\bigr)_{n+1}} = \frac{\bigl(-q^2;q^2\bigr)_{\infty}\bigl(q,q^4,q^5;q^5\bigr)_{\infty}}{\bigl(q^2;q^2\bigr)_{\infty}}
\end{equation*}
and
\begin{equation*}
	\sum_{n \geq 0} \frac{q^{n^2+n}}{\bigl(q^2;q^2\bigr)_n\bigl(-q;q^2\bigr)_{n}} = \frac{\bigl(-q^2;q^2\bigr)_{\infty}\bigl(q^2,q^3,q^5;q^5\bigr)_{\infty}}{\bigl(q^2;q^2\bigr)_{\infty}}.
\end{equation*}
The first of these appears in Slater's compendium of Rogers--Ramanujan type identities \cite[equation~(17)]{Sl2}, while the second appears there in an equivalent form \cite[equation~(99)]{Sl2}. Theorem~\ref{strangesignthm} generalizes
\begin{equation*}
	\sum_{n \geq 0} \frac{q^{\binom{n+1}{2}}\bigl(-q^2;q^2\bigr)_n}{(q)_n\bigl(q;q^2\bigr)_{n+1}} = (-q)_{\infty}^2\bigl(-q^4;q^4\bigr)_{\infty}
\end{equation*}
and
\begin{equation*}
	\sum_{n \geq 0} \frac{q^{\binom{n+1}{2}}\bigl(-1;q^2\bigr)_n}{(q)_n\bigl(q;q^2\bigr)_{n}} = (-q)_{\infty}^2\bigl(-q^2;q^4\bigr)_{\infty},
\end{equation*}
which are the cases $(a,b) =({\rm i}q,-{\rm i}q)$
and $({\rm i},-{\rm i})$ of the $q$-analogue of Gauss' second theorem (see~\cite[equa\-tion~(1.8)]{AndKummer} or \cite[Appendix II, equation~(II.11)]{GR}),
\begin{equation*}
	\sum_{n \geq 0} \frac{(a,b)_nq^{\binom{n+1}{2}}}{(q)_n\bigl(abq;q^2\bigr)_n} = \frac{\bigl(aq,bq;q^2\bigr)_{\infty}}{\bigl(q,abq;q^2\bigr)_{\infty}}.
\end{equation*}

\begin{proof}[Proof of Theorem \ref{q^2thm}]
	We will prove this theorem with $q\mapsto q^{1/2}$.
	We start with the seed pair relative to $q$ \cite[equation~(4.4)]{Wa1} determined by
	\begin{equation} \label{eqn:fn2}
		f_n = (-1)^nq^{\frac{1}{4}n(3n-1)},
	\end{equation}
	which satisfies \eqref{eqn:fcond},
	and
	\begin{equation*}
		\beta_n = \frac{1}{\bigl(q^2;q^2\bigr)_n\bigl(-q^{\frac{1}{2}};q\bigr)_n}.
	\end{equation*}
	The sum is obtained immediately from \eqref{eqn:mainid}, and on the product side we have
	\begin{align*}
		\sum_{n\in\ZZ}q^{m\binom{n}{2}+an}f_n =
		\sum_{n\in\ZZ}(-1)^nq^{\frac{2m+3}{4}n^2+(\frac{4a-2m-1}{4})n }
		 =\bigl( q^{a+\frac{1}{2}} , q^{m+1-a} , q^{m+\frac{3}{2}};q^{m+\frac{3}{2}}\bigr)_\infty,
	\end{align*}		
	by \eqref{eqn:jtp}.
\end{proof}

To prove Theorem \ref{strangesignthm}, we will need the following Bailey pair relative to $q$, which does not appear to be in the literature.

\begin{Lemma} \label{newpair}
	The pair $(\alpha_n,\beta_n)$ is a Bailey pair relative to $q$, where
	\begin{equation*}
		\alpha_n = \frac{(-1)^{\binom{n}{2}}q^{\binom{n}{2}}\bigl(1-q^{2n+1}\bigr)}{1-q}
	\qquad
	\text{and}
	\qquad
		\beta_n = \frac{\bigl(-1;q^2\bigr)_n}{(q)_{2n}}.
	\end{equation*}
\end{Lemma}

\begin{proof}
	We start with the Bailey pair relative to $1$ \cite[p.~470, line 7]{Sl1},
	\begin{equation*}
		A_n =
		\begin{cases}
			1 &\text{if $n=0$}, \\
			(-1)^rq^{2r^2}(q^r + q^{-r}) &\text{if $n = 2r, r >0$},\\
			(-1)^rq^{2r^2}\bigl(q^r - q^{3r+1}\bigr) &\text{if $n = 2r+1$}
		\end{cases}
	\qquad
	\text{and}
	\qquad
		B_n = \frac{\bigl(-1;q^2\bigr)_n}{(q)_{2n}}.
	\end{equation*}
	Using the case $b=0$ of Lemma \ref{1toqlemma}, we have that $(\alpha_n,\beta_n)$ is a Bailey pair relative to $q$, where~${\beta_n = B_n}$ and
	\begin{equation*}
		\alpha_n = \frac{\bigl(1-q^{2n+1}\bigr)q^{n^2}}{1-q} \sum_{j=0}^n q^{-j^2}A_j.
	\end{equation*}
	To prove the lemma, then it suffices to show that for all $n$ we have
	\begin{equation*}
		\sum_{j=0}^n q^{-j^2}A_j = (-1)^{\binom{n}{2}} q^{-\binom{n+1}{2}}.
	\end{equation*}
	This is easily established by mathematical induction.
\end{proof}

\begin{proof}[Proof of Theorem \ref{strangesignthm}]
	Corresponding to the $\alpha_n$ in the Bailey pair in Lemma \ref{newpair}, we have
	\begin{align} \label{eqn:fn3}
		f_n = (-1)^{\binom{n}{2}}q^{\binom{n}{2}},
	\end{align}	
	which can again be seen to satisfy \eqref{eqn:fcond}.
	We obtain the product using \eqref{eqn:jtp}
	\begin{align*}
		\sum_{n\in\ZZ}q^{m\binom{n}{2}+an}f_n =
		\sum_{n\in\ZZ}(-1)^{\binom{n}{2}}q^{\frac{m+1}{2}n^2+\frac{2a-m-1}{2}n}
		=\bigl(-q^{a} , q^{m+1-a} , -q^{m+1};-q^{m+1}\bigr)_\infty.
	\end{align*}
	Here, while invoking \eqref{eqn:jtp} with
	$q\mapsto {\rm i}q^{\frac{m+1}{2}}$ and $z\mapsto {\rm i}q^{\frac{2a-m-1}{2}}$,
	we use that
	\begin{align*}
		(-1)^{n}{\rm i}^{n+n^2}=(-1)^{n+\frac{n^2+n}{2}}
	=(-1)^{\binom{n}{2}}.
\tag*{\qed}
\end{align*}
\renewcommand{\qed}{}
\end{proof}

The final result in this section collects three further applications of Theorem \ref{thm:main}, where the infinite product now arises from the quintuple product identity \eqref{eqn:qtp} instead of the triple product identity.

\begin{Theorem}
\label{thm:quintuple}
For $m\geq 1$ and $0\leq a\leq m$, we have the following identities:
\begin{gather}
		\sum_{n_m, \dots, n_1 \geq 0} \frac{q^{\binom{n_m + 1}{2} + \cdots + \binom{n_1+1}{2}}}{(q)_{n_m}\bigl(q;q^2\bigr)_{n_1 + \delta_{a,0}}} \prod_{i=1}^{m-1}\begin{bmatrix} n_{i+1} + \delta_{a,i} \\ n_i \end{bmatrix} \nonumber\\
		\qquad{}= \dfrac{(-q)_\infty}{(q)_\infty}\bigl(q^{m+1-a},q^{2m+3+a},q^{3m+4};q^{3m+4}\bigr)_{\infty}\bigl(q^{m+2+2a},q^{5m+6-2a};q^{6m+8}\bigr)_{\infty},
\label{eqn:quintex1} \\
		\sum_{n_m,\dots ,n_1 \geq 0} \frac{q^{\binom{n_m + 1}{2} + \cdots + \binom{n_1+1}{2}+ (n_1+\delta_{a,0})^2-(n_1+\delta_{a,0})}}{(q)_{n_m} \bigl(q;q^2\bigr)_{n_1 + \delta_{a,0}}} \prod_{i=1}^{m-1} 
		\begin{bmatrix} n_{i+1} + \delta_{a,i} \\ n_i \end{bmatrix}
		\nonumber \\
		\qquad{}=
		\dfrac{(-q)_\infty}{(q)_\infty}
		\bigl(q^{m+1-a}, q^{2m+1+a}, q^{3m+2} ; q^{3m+2}\bigr)_\infty \bigl(q^{m+2a},q^{-2a+5m+4} ;q^{6m+4}\bigr)_\infty,
\label{eqn:quintex2} \\
		\sum_{n_m,\dots ,n_1 \geq 0}
		\frac{q^{\binom{n_m + 1}{2} + \cdots + \binom{n_1+1}{2}}\bigl(-1;q^3\bigr)_{n_1+\delta_{a,0}}}
		{(q)_{n_m} \bigl(q;q^2\bigr)_{n_1 + \delta_{a,0}}(-1;q)_{n_1+\delta_{a,0}}} \prod_{i=1}^{m-1} 
		\begin{bmatrix} n_{i+1} + \delta_{a,i} \\ n_i \end{bmatrix} \nonumber\\
		\qquad{}=
		\dfrac{(-q)_\infty}{(q)_\infty}
		\bigl(q^{m+1-a}, q^{2m+2+a}, q^{3m+3} ; q^{3m+3}\bigr)_\infty \bigl(q^{m+2a+1},q^{-2a+5m+5} ;q^{6m+6}\bigr)_\infty.
\label{eqn:quintex3}
\end{gather}

\end{Theorem}
\begin{proof}
For each of these identities, we will use a Bailey pair with relative to $q$ that has the form \eqref{eqn:fdef} with $f_n$ of the following shape for some $t$:
\begin{align} \label{eqn:fnquint}
	f_n = \begin{cases}
		q^{\frac{t}{3}\binom{n+1}{2}-n} & \text{if $n\equiv 0 \pmod{3}$},\\
		0 & \text{if $n \equiv 1 \pmod{3}$},\\
		-q^{\frac{t}{3}\binom{n+1}{2}-n} & \text{if $n \equiv 2 \pmod{3}$}.\\
	\end{cases}
\end{align}
It can be seen that this $f_n$ satisfies \eqref{eqn:fcond} for any $t$, however, for our identities, we will only require $t\in\{2,3,4\}$.

For \eqref{eqn:quintex1}, we use $t=4$, and the corresponding Bailey pair is \cite[equation~(4.6)]{Wa1}, with
\begin{equation} \label{eqn:quintbeta1}
	\beta_n = \frac{1}{(q)_{2n}}.
\end{equation}
For \eqref{eqn:quintex2}, we use $t=2$, and the Bailey pair is given right before \cite[Theorem~4.1]{Wa1}, with
\begin{equation} \label{eqn:quintbeta2}
	\beta_n = \frac{q^{n^2-n}}{(q)_{2n}}.
\end{equation}
For \eqref{eqn:quintex3}, we use $t=3$, and the Bailey pair is \cite[P(6)]{McLS}. This can also be obtained from the pair $P(1)$ of \cite{McLS}, by using Lemma \ref{1toqlemma} to retain the $\beta_n$ and change the $\alpha_n$ so that $x=q$; see also \cite{KR}. This pair has
\begin{align} \label{eqn:quintbeta3}
	\beta_n = \dfrac{\bigl(-1;q^3\bigr)_n}{(q)_{2n}(-1;q)_n}.
\end{align}

Using \eqref{eqn:quintbeta1} and \eqref{eqn:quintbeta2}, and \eqref{eqn:quintbeta3} in \eqref{eqn:mainid}, we immediately obtain the requisite sum-sides of~\eqref{eqn:quintex1}--\eqref{eqn:quintex3}. For the products, using \eqref{eqn:fnquint} we obtain
\begin{align*}
	\sum_{n\in\ZZ}
	q^{m\binom{n}{2}+an}f_n
	&{}=
	\Biggl(\sum_{\substack{n\in\ZZ\\ n\equiv 0\,\,(\mathrm{mod}\,\,3)}}
	-\sum_{\substack{n\in\ZZ\\ n\equiv 2\,\,(\mathrm{mod}\,\,3)}}\Biggr)
	q^{m\binom{n}{2}+an+\frac{t}{3}\binom{n+1}{2}-n} \\
	&{}=
	\Biggl(\sum_{\substack{n\in\ZZ\\ n\equiv 0\,\,(\mathrm{mod}\,\,3)}}
	-\sum_{\substack{n\in\ZZ\\ n\equiv 2\,\,(\mathrm{mod}\,\,3)}}\Biggr)
	q^{\frac{3m+t}{3}\binom{n+1}{2}+(-m+a-1)n},
\end{align*}
and then substituting $(z,q)\mapsto \bigl(q^{-m+a-1},q^{3m+t}\bigr)$ in \eqref{eqn:qtp} completes the proof.
\end{proof}

\begin{Remark}
The base cases of \eqref{eqn:quintex1}, corresponding to $m=1$ and $a=0,1$ are recorded with the right-hand side in the form of a theta series by Rogers \cite[p.~331, equation~(1)]{Ro1}. These same cases of \eqref{eqn:quintex2} are also given by Rogers \cite[p.~330, equation~(2)]{Ro1}.
The cases $m=1$ and $a=0,1$ of \eqref{eqn:quintex3} could be deduced by using the limiting case of the Bailey lemma (with $b=-q$, $c\rightarrow\infty$) applied to the Bailey pairs relative to $q$ given in \cite[equations~(3.6) and~(3.11)]{And1}, respectively.
\end{Remark}

\begin{Remark}
Note that the products in all three of these identities are closely related to the principal characters of standard \smash{$A_2^{(2)}$} modules. Alternate sum-sides for the same products may be found in \cite{KR}.
It would be interesting to investigate whether the sums appearing in Theorem~\ref{thm:quintuple} have any connection to the Ariki--Koike algebras or to Kleshchev multipartitions.
\end{Remark}

\section{Proof of Theorem \ref{thm:false} and applications}
\label{sec:falsetheta}

We begin this section by establishing a key Bailey lemma by following steps similar to Steps 1--6 in Section \ref{sec:mainthm} with appropriate modifications. Only slight modifications are required for $a=m$ or when $a < m-1$, but for the case $a=m-1$ we will use an analogue of the case $x=q$ of Lemma \ref{littlelemma} with $\beta_n' = \bigl(1-q^{n+1}\bigr)\beta_{n+1}$ replaced by $\beta_n' = \bigl(1+q^{n+1}\bigr)\beta_{n+1}$. This will be the third of a sequence of three lemmas. The first of these is a transcription of \cite[Lemma 3.3]{Lo1}.
\begin{Lemma} \label{falselemma1}
If $(\alpha_n,\beta_n)$ is a Bailey pair relative to $x$ with $\alpha_0 = \beta_0 = 0$, then $(\alpha_n',\beta_n')$ is a~Bailey pair relative to $xq$, where
\begin{equation*}
	\alpha_n' = \frac{1}{1-xq} \biggl( \frac{1}{1-xq^{2n+2}}\alpha_{n+1} - \frac{xq^{2n}}{1-xq^{2n}}\alpha_n \biggr)
\qquad
\text{and}
\qquad
	\beta_n' = \beta_{n+1}.
\end{equation*}
\end{Lemma}
The next lemma removes the hypothesis $\alpha_0 = \beta_0 = 0$ when $x=q$.
\begin{Lemma} \label{falselemma2}
If $(\alpha_n,\beta_n)$ is a Bailey pair relative to $q$, then $(\alpha_n',\beta_n')$ is a Bailey pair relative~to~$q^2$, where
\begin{equation*}
	\alpha_n' = \frac{1}{1-q^2} \biggl( \frac{1}{1-q^{2n+3}}\alpha_{n+1} - \frac{q^{2n+1}}{1-q^{2n+1}}\alpha_n \biggr) + \frac{\beta_0}{(q)_2} \bigl(1+q^{n+1}\bigr)(-1)^nq^{\binom{n+1}{2}}
\end{equation*}
and
\begin{equation*}
	\beta_n' = \beta_{n+1}.
\end{equation*}
\end{Lemma}

\begin{proof}
Recall the unit Bailey pair relative to $q$ \cite[equations~(2.12) and (2.13), $a \mapsto q$]{Andmultiple}:
\begin{equation*}
	A_n = \frac{1-q^{2n+1}}{1-q}(-1)^nq^{\binom{n}{2}}
\qquad
\text{and}
\qquad
	B_n = \delta_{n,0}.
\end{equation*}
Using the linearity of Bailey pairs, we have that $(\alpha_n - A_n\beta_0,\beta_n - B_n\beta_0)$ is a Bailey pair relative to $q$ satisfying the hypothesis of Lemma \ref{falselemma1}. A short computation then gives the result.
\end{proof}

The third lemma is the necessary lemma for our proof of the case $a=m-1$.
\begin{Lemma} \label{falselemma3}
If $(\alpha_n,\beta_n)$ is a Bailey pair relative to $q$, then $(\alpha_n',\beta_n')$ is a Bailey pair relative~to~$q^2$, where
\begin{equation*} 
	\alpha_n' = \frac{1}{1-q^2} \biggl( \frac{1+q^{n+1}}{1-q^{2n+3}}\alpha_{n+1} - \frac{q^n\bigl(1+q^{n+1}\bigr)}{1-q^{2n+1}}\alpha_n \biggr) + \frac{2\beta_0}{(q)_2} \bigl(1+q^{n+1}\bigr)(-1)^nq^{\binom{n+1}{2}}
\end{equation*}
and
\begin{equation*} 
	\beta_n' = \bigl(1+q^{n+1}\bigr)\beta_{n+1}.
\end{equation*}
\end{Lemma}

\begin{proof}
The result follows from taking the negative of the Bailey pair in Lemma \ref{littlelemma} with $x=q$ plus twice the Bailey pair in Lemma \ref{falselemma2}.
\end{proof}

We are now ready to prove the key result on Bailey pairs.
\begin{Lemma} \label{lemma:falsepair}
Let $m \geq 1$ and $0 \leq a \leq m$. Suppose $(\alpha_n,\beta_n)$ is a Bailey pair relative to $q$ and let~$f_n$ be as in Lemma $\ref{lemma:mainpair}$. Then $(\alpha_n',\beta_n')$ is a Bailey pair relative to $q$, where
\begin{equation*} 
	\alpha_n' =
	\begin{cases}
		\dfrac{1}{1-q}(-1)^nq^{m\binom{n+1}{2}}\bigl(
		q^{a(n+1)}f_{n+1} -q^{-an+2n-1}f_{n-1}
		\bigr) & \text{if $0 \leq a \leq m-1$}, \\
		\dfrac{1}{1-q}(-1)^nq^{m \binom{n+1}{2}}\bigl(1-q^{2n+1}\bigr)f_n &\text{if $a=m$},
	\end{cases}
\end{equation*}
and
\begin{align}
		\beta_n' &{}= \beta_{n_{m+1}}' \label{eqn:falsepairbeta}\\
 &{}= \frac{1}{(q)_{n_{m+1}}} \sum_{n_m, \dots, n_1 \geq 0}\frac{(-1)^{n_m}q^{\binom{n_m+1}{2} + \cdots + \binom{n_1+1}{2}}\bigl(q^2;q^2\bigr)_{n_1+\delta_{a,0}}}{(q)_{n_{m+1}-n_m} (-q)_{n_m}} \beta_{n_1+\delta_{a,0}} \prod_{i=1}^{m-1} \begin{bmatrix} n_{i+1} + \delta_{a,i} \\ n_i \end{bmatrix}.
\nonumber
\end{align}
\end{Lemma}

\begin{proof}
The proof is broken up into three cases. First suppose that $a=m$. In this case, we apply Step 1 as in the proof of Lemma \ref{lemma:mainpair}, except that at the final iteration of Lemma \ref{Baileylemma} we use $b \mapsto q$ and $c \to \infty$ instead of $b \mapsto -q$ and $c \to \infty$. The resulting pair is
\begin{align*}
	\alpha_n= (-1)^nq^{m\binom{n+1}{2}}\dfrac{1-q^{2n+1}}{1-q}f_n
\end{align*}
and
\begin{align*}
	\beta_n=\beta_{n_{m+1}}=
	\dfrac{1}{(q)_{n_{m+1}}}\sum_{n_m,\dots,n_1\geq 0}
	\dfrac{(q)_{n_m}(-1)^{n_m}q^{\binom{n_m+1}{2}+\cdots+\binom{n_1+1}{2}}(-q)_{n_1}}{(-q)_{n_m}(q)_{n_{m+1}-n_m}\cdots (q)_{n_2-n_1}}\beta_{n_1}.
\end{align*}
Multiplying the numerator and denominator of the above sum by $(q)_{n_1} \cdots (q)_{n_{m-1}}$ and converting to $q$-binomial coefficients gives \eqref{eqn:falsepairbeta}.

Next assume that $a < m-1$. In this case, we follow the same steps as in the proof of Lemma~\ref{lemma:mainpair} except that in the very last iteration of Step 5 we use $b = q$ and $ c\to \infty$ instead of $b = -q$ and $c \to \infty$. The result is
\begin{equation*}
	\alpha_n =
	\dfrac{1}{1-q}(-1)^nq^{m\binom{n+1}{2}}\bigl(
	q^{a(n+1)}f_{n+1} -q^{-an+2n-1}f_{n-1}
	\bigr)
\end{equation*}
and either
\begin{equation*}
	\beta_n=\beta_{n_{m+1}}=
	\dfrac{1}{(q)_{n_{m+1}}}\sum_{n_m,\dots,n_1\geq 0}\!
	\dfrac{(q)_{n_m}(-1)^{n_m}q^{\binom{n_{m}+1}{2}+\cdots+\binom{n_1+1}{2}}\bigl(1-q^{n_{a+1}+1}\bigr)(-q)_{n_1}}
	{(-q)_{n_m}(q)_{n_{m+1}-n_{m}}\cdots
		(q)_{n_{a+1}+1-n_a}\cdots (q)_{n_2-n_1}}\beta_{n_1}
\end{equation*}
if $a > 0$, or
\begin{gather*}
	\beta_n=\beta_{n_{m+1}}=
	\dfrac{1}{(q)_{n_{m+1}}}\sum_{n_m,\dots,n_1\geq 0}\!
	\dfrac{(q)_{n_m}(-1)^{n_m}q^{\binom{n_{m}+1}{2}+\cdots+\binom{n_1+1}{2}}\bigl(1-q^{n_{1}+1}\bigr)(-q)_{n_1+1}}
	{(-q)_{n_m}(q)_{n_{m+1}-n_{m}}\cdots (q)_{n_2-n_1}}
	\beta_{n_1+1}
\end{gather*}
if $a=0$. Rewriting in terms of $q$-binomial coefficients gives the result.

Finally, we consider the case $a = m-1$. If $a > 0$, then we begin with the Bailey pair in \eqref{eqn:alphafn} and \eqref{eqn:betastep1}, then redo Steps 2--4 with modifications as follows.

{\it Step 2\emph{:}} Applying Lemma \ref{falselemma3}, we obtain a Bailey pair relative to $q^2$,
\begin{align*}
	\alpha_n &{}= \frac{1+q^{n+1}}{(q)_2} \bigl(q^{a\binom{n+2}{2}}f_{n+1} - q^{n + a\binom{n+1}{2}}f_n \bigr) + 2\frac{\bigl(1+q^{n+1}\bigr)}{(q)_2}(-1)^nq^{\binom{n+1}{2}} \\
	&{}= \frac{1+q^{n+1}}{(q)_2}q^{(a+2)\binom{n+1}{2} - n^2} \bigl(-f_n + q^{a(n+1)-n}f_{n+1} \bigr) + 2\frac{\bigl(1+q^{n+1}\bigr)}{(q)_2}(-1)^nq^{\binom{n+1}{2}}
\end{align*}
and
\begin{equation*}
	\beta_n=\beta_{n_{a+1}}=
	\dfrac{1}{(-q)_{n_{a+1}}}\sum_{n_a,\dots,n_1\geq 0}
	\dfrac{q^{\binom{n_a+1}{2}+\cdots+\binom{n_1+1}{2}}(-q)_{n_1}}{(q)_{n_{a+1}+1-n_a}\cdots (q)_{n_2-n_1}}\beta_{n_1}.
\end{equation*}

{\it Step 3\emph{:}} Using Lemma \ref{Baileylemma} with $x \mapsto q^2$, $b \mapsto q^2$ and $c \to \infty$, we obtain another Bailey pair relative to $q^2$,
\begin{equation*}
	\alpha_n = \frac{1-q^{2n+2}}{(1-q)(q)_2}(-1)^nq^{(a+3)\binom{n+1}{2} - n^2} \bigl(-f_n + q^{a(n+1)-n}f_{n+1} \bigr) + 2\frac{\bigl(1-q^{2n+2}\bigr)}{(1-q)(q)_2}q^{n^2+n}
\end{equation*}
and
\begin{gather*}
	\beta_n=\beta_{n_{a+2}}=
	\dfrac{1}{(1-q)(q)_{n_{a+2}}}\sum_{n_{a+1},\dots,n_1\geq 0}\!
	\dfrac{(-1)^{n_{a+1}}(q)_{n_{a+1}+1}q^{\binom{n_{a+1}+1}{2}+\cdots+\binom{n_1+1}{2}}(-q)_{n_1}}{(-q)_{n_{a+1}}(q)_{n_{a+2} - n_{a+1}}(q)_{n_{a+1}+1-n_a}\cdots (q)_{n_2-n_1}}\beta_{n_1}.
\end{gather*}
The factor $(1-q)$ cancels.

{\it Step 4\emph{:}} Use Lemma \ref{Baileylatticelemma} with $x \mapsto q^2$ and $b,c \mapsto q$ to obtain a Bailey pair relative to $q$. This does not change the $\beta_n$, which, after rewriting in terms of $q$-binomial coefficients, coincides with~\eqref{eqn:falsepairbeta} for $a = m-1$ and $a>0$. As for the $\alpha_n$, for $n > 0$ we have
\begin{gather*}
	\alpha_n = \frac{1}{1-q}\bigl((-1)^nq^{(a+3)\binom{n+1}{2} - n^2}\bigl(-f_n + q^{a(n+1)-n}f_{n+1}\bigr)\bigr) + \frac{2q^{n^2+n}}{1-q} \\
	\hphantom{\alpha_n =}{} - \frac{1}{1-q}\bigl((-1)^{n-1}q^{2n+(a+3)\binom{n}{2} - (n-1)^2}\bigl(-f_{n-1} + q^{an-(n-1)}f_n\bigr)\bigr) - \frac{2q^{2n}q^{n^2-n}}{1-q}\\
	\hphantom{\alpha_n}{} = \frac{1}{1-q}\bigl((-1)^nq^{(a+1)\binom{n+1}{2}}\bigl(q^{a(n+1)}f_{n+1} - q^{-an+2n-1}f_{n-1}\bigr)\bigr).
\end{gather*}
Using \eqref{eqn:latticeextracond} and \eqref{eqn:f-1f0}, the above is also valid for $n=0$. This completes the proof except for ${m=1}$, ${a=0}$. For this case, the argument is similar. One applies Steps 2--4 beginning with the Bailey pair $(\alpha_n,\beta_n)$ in the statement of the lemma.
\end{proof}

With the Bailey pair from Lemma \ref{lemma:falsepair} in hand, we are now ready to deduce Theorem \ref{thm:false}.

\begin{proof}[Proof of Theorem \ref{thm:false}]
Let $(\alpha_n,\beta_n)$ be a Bailey pair relative to $q$ satisfying \eqref{eqn:fdeflater} and \eqref{eqn:f-1f0}.
Applying \eqref{Baileylimit} to the Bailey pair in Lemma~\ref{lemma:falsepair}, we obtain
\begin{gather}
		\sum_{n_m, \dots, n_1 \geq 0} \frac{(-1)^{n_m}q^{\binom{n_m+1}{2} + \cdots + \binom{n_1+1}{2}}\bigl(q^2;q^2\bigr)_{n_1+\delta_{a,0}}}{(-q)_{n_m}} \beta_{n_1+\delta_{a,0}} \prod_{i=1}^{m-1} \begin{bmatrix} n_{i+1} + \delta_{a,i} \\ n_i \end{bmatrix}
		\nonumber\\
		\qquad{}=
		\begin{cases}
			\displaystyle\sum_{n\geq 0}
			(-1)^nq^{m\binom{n+1}{2}}\bigl(
			q^{a(n+1)}f_{n+1} -q^{-an+2n-1}f_{n-1}\bigr)
			& \text{if $0\leq a \leq m-1$},\vspace{1mm}\\
			\displaystyle\sum_{n\geq 0}
			(-1)^nq^{m\binom{n+1}{2}}\bigl(
			1-q^{2n+1}\bigr)f_n
			& \text{if $a=m$}.
		\end{cases}\label{eqn:falseintermediate}
\end{gather}
If $f_n$ also satisfies \eqref{eqn:fcond}, then a short computation as in the proof of Theorem~\ref{thm:main} converts the right-hand side of \eqref{eqn:falseintermediate} into the right-hand side of~\eqref{eqn:falseid}.
\end{proof}

Using Theorem \ref{thm:false}, we obtain false theta companions for each family of identities established in Section \ref{sec:mainthm}. The proofs use the same Bailey pairs but in \eqref{eqn:falseid} instead of \eqref{eqn:mainid}. For example, as noted in the introduction, Theorem \ref{thm:falseCLSXY} follows upon inserting \eqref{seed1alpha} and \eqref{seed1beta} in \eqref{eqn:falseid}.

The false theta companions to \eqref{q^2eq} and \eqref{strangesignex}, obtained using the Bailey pairs corresponding to \eqref{eqn:fn2} and \eqref{eqn:fn3} in Theorem \ref{thm:false} are collected in the following two results.
\begin{Theorem} \label{thm:falseq^2}
For $m \geq 1$ and $0 \leq a \leq m$, we have
\begin{gather*}
		\sum_{n_m,\dots ,n_1 \geq 0} \frac{(-1)^{n_m}q^{n_m^2+n_m + \cdots + n_1^2+n_1}}{\bigl(-q^2;q^2\bigr)_{n_m} \bigl(-q;q^2\bigr)_{n_1 + \delta_{a,0}}} \prod_{i=1}^{m-1} \begin{bmatrix} n_{i+1} + \delta_{a,i} \\ n_i \end{bmatrix}_{q^2} \\
		\qquad{}=
		\begin{cases}
			\displaystyle \sum_{n \in \mathbb{Z}} {\rm sgn}(-n)q^{m(n^2-n) + n(3n-1)/2 +2an} &\text{if $0 \leq a \leq m-1$}, \vspace{1mm}\\
			\displaystyle \sum_{n \in \mathbb{Z}} {\rm sgn}(n)q^{m(n^2+n) + n(3n-1)/2} & \text{if $a=m$.}
		\end{cases}
\end{gather*}
\end{Theorem}

\begin{Theorem} \label{thm:falsestrangesign}
For $m \geq 1$ and $0 \leq a \leq m$, we have
\begin{gather*}
		\sum_{n_m, \dots, n_1 \geq 0} \frac{(-1)^{n_m}q^{\binom{n_m + 1}{2} + \cdots + \binom{n_1+1}{2}}\bigl(-1;q^2\bigr)_{n_1 + \delta_{a,0}}}{(-q)_{n_m}\bigl(q;q^2\bigr)_{n_1+ \delta_{a,0}}} \prod_{i=1}^{m-1} \begin{bmatrix} n_{i+1} + \delta_{a,i} \\ n_i \end{bmatrix} \\
		\qquad{}=
		\begin{cases}
			\displaystyle \sum_{n \in \mathbb{Z}} {\rm sgn}(-n)(-1)^{\binom{n+1}{2}} q^{(m+1)\binom{n}{2} + an} & \text{if $0 \leq a \leq m-1$}, \vspace{1mm}\\
			\displaystyle \sum_{n \in \mathbb{Z}} {\rm sgn}(n)(-1)^{\binom{n+1}{2}} q^{(m+1)\binom{n+1}{2}-n} & \text{if $a=m$}.
		\end{cases}
	\end{gather*}
\end{Theorem}

Finally, we record the false theta companions to the identities in Theorem \ref{thm:quintuple}. To keep the expressions concise, we use the Legendre symbol
\begin{equation*}
\Bigl(\frac{n}{3}\Bigr) =
\begin{cases}
	\hphantom{-}1 & \text{if $n \equiv 1 \pmod{3}$}, \\
	-1 & \text{if $n \equiv -1 \pmod{3}$}, \\
	\hphantom{-}0 & \text{if $n \equiv 0 \pmod{3}$}.
\end{cases}
\end{equation*}

\begin{Theorem} \label{thm:falsequintuple}
For $m \geq 1$ and $0 \leq a \leq m$, we have
\begin{gather}
		\sum_{n_m, \dots, n_1 \geq 0} \frac{(-1)^{n_m}q^{\binom{n_m + 1}{2} + \cdots + \binom{n_1+1}{2}}}{(-q)_{n_m}\bigl(q;q^2\bigr)_{n_1 + \delta_{a,0}}} \prod_{i=1}^{m-1} \begin{bmatrix} n_{i+1} + \delta_{a,i} \\ n_i \end{bmatrix}\nonumber \\
		\qquad{}=
		\begin{cases}
			\displaystyle \sum_{n \in \mathbb{Z}} {\rm sgn}(-n)(-1)^n\biggl(\frac{1-n}{3}\biggr) q^{\frac{3m+4}{3}\binom{n+1}{2} + (a-m-1)n} & \text{if $0 \leq a \leq m-1$}, \\
			\displaystyle \sum_{n \in \mathbb{Z}} {\rm sgn}(n)(-1)^n\biggl(\frac{1-n}{3}\biggr) q^{\frac{3m+4}{3}\binom{n+1}{2} - n} & \text{if $a=m$},
		\end{cases}
 \label{eqn:falsequintupleeq1}\\
		\sum_{n_m, \dots, n_1 \geq 0} \frac{(-1)^{n_m}q^{\binom{n_m + 1}{2} + \cdots + \binom{n_1+1}{2} + (n_1+\delta_{a,0})^2 - (n_1 + \delta_{a,0})}}{(-q)_{n_m}\bigl(q;q^2\bigr)_{n_1 + \delta_{a,0}}} \prod_{i=1}^{m-1} \begin{bmatrix} n_{i+1} + \delta_{a,i} \\ n_i \end{bmatrix} \nonumber \\
		\qquad{}=
		\begin{cases}
			\displaystyle \sum_{n \in \mathbb{Z}} {\rm sgn}(-n)(-1)^n\biggl(\frac{1-n}{3}\biggr) q^{\frac{3m+2}{3}\binom{n+1}{2} + (a-m-1)n} & \text{if $0 \leq a \leq m-1$}, \vspace{1mm}\\
			\displaystyle \sum_{n \in \mathbb{Z}} {\rm sgn}(n)(-1)^n\biggl(\frac{1-n}{3}\biggr) q^{\frac{3m+2}{3}\binom{n+1}{2} - n} & \text{if $a=m$},
		\end{cases}
\label{eqn:falsequintupleeq2} \\
		\sum_{n_m, \dots, n_1 \geq 0} \frac{(-1)^{n_m}q^{\binom{n_m + 1}{2} + \cdots + \binom{n_1+1}{2}}\bigl(-1;q^3\bigr)_{n_1+\delta_{a,0}}}{(-q)_{n_m}\bigl(q;q^2\bigr)_{n_1 + \delta_{a,0}}(-1;q)_{n_1 + \delta_{a,0}}} \prod_{i=1}^{m-1} \begin{bmatrix} n_{i+1} + \delta_{a,i} \\ n_i \end{bmatrix} \nonumber\\
		\qquad{}=
		\begin{cases}
			\displaystyle \sum_{n \in \mathbb{Z}} {\rm sgn}(-n)(-1)^n\biggl(\frac{1-n}{3}\biggr) q^{(m+1)\binom{n+1}{2} + (a-m-1)n} & \text{if $0 \leq a \leq m-1$}, \\
			\displaystyle \sum_{n \in \mathbb{Z}} {\rm sgn}(n)(-1)^n\biggl(\frac{1-n}{3}\biggr) q^{(m+1)\binom{n+1}{2} - n} & \text{if $a=m$}.
		\end{cases} \nonumber
\end{gather}

\end{Theorem}

We close this section by noting that a number of the bases cases of Theorems \ref{thm:falseq^2}--\ref{thm:falsequintuple} appear in classical work of Rogers \cite{Ro1}. For example, the base cases corresponding to $m=1$ of Theorem~\ref{thm:falseq^2} are found at \cite[p.~334, equation~(7)]{Ro1}. The case $m=1$ and $a=0$ of Theorem~\ref{thm:falsestrangesign} is found at \cite[p.~333, equation~(5)]{Ro1}. Finally, the two cases of \eqref{eqn:falsequintupleeq1} and \eqref{eqn:falsequintupleeq2} corresponding to~${m=1}$ are at \cite[p.~332, equation~(1)]{Ro1} and \cite[p.~332, equation~(2)]{Ro1}, respectively.

\section{Proof of Theorem \ref{thm:dilation} and applications}
\label{sec:dilated}

We begin this section with the following Bailey lemma.

\begin{Lemma} \label{lemma:dilationpair}
Let $m \geq 1$ and $0 \leq a \leq m$. Suppose that $(\alpha_n,\beta_n)$ is a Bailey pair relative to $1$ and that there is a sequence $(g_n)_{n \geq 0}$ with $g_0 = 1$ and
\begin{equation*}
	\alpha_n =
	\begin{cases}
		1 &\text{if $n= 0$}, \\
		(1+q^n)g_n &\text{if $n > 1$}.
	\end{cases}
\end{equation*}
Then $(\alpha_n',\beta_n')$ is a Bailey pair relative to $1$, where
\begin{gather} \label{eqn:dilationpairalpha}
	\alpha_n' =
	\begin{cases}
		\dfrac{1+q^{\frac{a}{2}}g_1}{1-q^{\frac{1}{2}}} &\text{if $n = 0$}, \vspace{1mm}\\
		q^{\frac{m}{2}n^2+n}\left( \dfrac{\bigl(g_n + q^{n(a-1)+ \frac{a}{2}}g_{n+1}\bigr)}{1-q^{n+\frac{1}{2}}} - \dfrac{q^{-\frac{1}{2}}\bigl(g_n + q^{-n(a-1) + \frac{a-2}{2}}g_{n-1}\bigr)}{1-q^{n - \frac{1}{2}}} \right) &\text{if $n > 0$},
	\end{cases}\!\!\!\!
\end{gather}
and
\begin{align}
		\beta_n' &{}=\beta_{n_{m+1}} \label{eqn:dilationpairbeta}\\
		&{}=
		\frac{1}{\bigl(-q^{\frac{1}{2}},q\bigr)_{n_{m+1}}}
		\sum_{n_{m},\dots,n_1\geq 0}
		q^{\frac{1}{2}(n_{m}^2+n_{m-1}^2+\cdots+n_1^2)}\bigl(-q^{\frac{1}{2}},q\bigr)_{n_1 + \delta_{a,0}}\beta_{n_1 + \delta_{a,0}}
		 \prod_{i=1}^{m} \begin{bmatrix} n_{i+1} + \delta_{a,i} \\ n_i \end{bmatrix}.\nonumber
\end{align}
\end{Lemma}

\begin{proof}
As in the two previous sections, we break the proof into several steps.

{\it Step 1\emph{:}} We first assume that $a>0$, postponing the case $a=0$ case until later (see Step 6).
Iterating Lemma \ref{Baileylemma} $a$ times with $x\mapsto 1$, \smash{$b\mapsto -q^{\frac{1}{2}}$}, and $c\rightarrow \infty$, we obtain the following pair relative to $1$:
\begin{align}
	\alpha_n =
	\begin{cases}
		1 & \text{if $n=0$}, \\
		q^{\frac{a}{2}n^2}(1+q^n)g_n & \text{if $n>0$},
	\end{cases}
	\label{eqn:alphathm1.6step1}
\end{align}
and
\begin{align}
	\beta_n =\beta_{n_{a+1}}
	=\dfrac{1}{\bigl(-q^{\frac{1}{2}}\bigr)_{n_{a+1}}}\sum_{n_{a},\dots,n_1\geq 0}
	\dfrac{q^{\frac{1}{2}(n_{a}^2+n_{a-1}^2+\cdots+n_1^2)}\bigl(-q^{\frac{1}{2}}\bigr)_{n_1}}{(q)_{n_{a+1}-n_{a}}
		(q)_{n_{a}-n_{a-1}}
		\cdots
		(q)_{n_{2}-n_{1}}} \beta_{n_1}.
	\label{eqn:betathm1.6step1}
\end{align}

{\it Step 2\emph{:}}
Now we apply Lemma \ref{littlelemma} with $x \mapsto 1$ to obtain a Bailey pair relative to $q$:
\begin{align*}
	\alpha_{n} &{}=
	\begin{cases}
		\dfrac{1 + q^{\frac{a}{2}}g_1}{1-q} 	 & \text{if $n=0$}, \\
		\dfrac{q^{\frac{a}{2}n^2 + n}}{1-q}\bigl(g_n + q^{n(a-1) + \frac{a}{2}}g_{n+1} \bigr) & \text{if $n > 0$}
	\end{cases}
	\\
	&{}=
	\dfrac{q^{\frac{a}{2}n^2 + n}}{1-q}\bigl(g_n + q^{n(a-1) + \frac{a}{2}}g_{n+1} \bigr),
\end{align*}
and
\begin{align*}
	\beta_n =\beta_{n_{a+1}}
	=\dfrac{\bigl(1-q^{n_{a+1}+1}\bigr)}{\bigl(-q^{\frac{1}{2}}\bigr)_{n_{a+1}+1}}\sum_{n_{a},\dots,n_1\geq 0}
	\dfrac{q^{\frac{1}{2}(n_{a}^2+n_{a-1}^2+\cdots+n_1^2)}\bigl(-q^{\frac{1}{2}}\bigr)_{n_1}}{(q)_{n_{a+1}+1-n_{a}}
		(q)_{n_{a}-n_{a-1}}
		\cdots
		(q)_{n_{2}-n_{1}}} \beta_{n_1}.
\end{align*}

{\it Step 3\emph{:}}
Next we apply Lemma \ref{Baileylemma} with $x\mapsto q$, \smash{$b\mapsto -q^{\frac{3}{2}}$}, and $c\rightarrow \infty$, and then cancel a factor of \smash{$\bigl(1+q^{\frac{1}{2}}\bigr)^{-1}$} in the resulting $\alpha_n$ and $\beta_n$ to obtain the following Bailey pair relative to~${x=q}$:
\begin{align*}
	\alpha_{n} =
	\dfrac{q^{\frac{a+1}{2}n^2 + n}}{1-q}\bigl(g_n + q^{n(a-1) + \frac{a}{2}}g_{n+1} \bigr)\bigl(1+q^{n+ \frac{1}{2}}\bigr)
\end{align*}
and
\begin{align*}
	\beta_n =\beta_{n_{a+2}}
	=
	\frac{1}{ \bigl(-q^{\frac{1}{2}}\bigr)_{n_{a+2}}}
	\sum_{n_{a+1},\dots,n_1\geq 0}
	\dfrac{\bigl(1-q^{n_{a+1}+1}\bigr)q^{\frac{1}{2}(n_{a+1}^2+n_{a}^2+\cdots+n_1^2)}\bigl(-q^{\frac{1}{2}}\bigr)_{n_1}}{(q)_{n_{a+2}-n_{a+1}}(q)_{n_{a+1}+1-n_{a}}
		(q)_{n_{a}-n_{a-1}}
		\cdots
		(q)_{n_{2}-n_{1}}} \beta_{n_1}.
\end{align*}

{\it Step 4\emph{:}}	In the next step, we retain the $\beta_n$, but change $\alpha_n$ so that the new Bailey pair is relative to $x=1$, by using Lemma \ref{Baileylatticelemma} with $x\mapsto q$ and $b,c\mapsto q^{\frac{1}{2}}$. Keeping in mind \eqref{eqn:latticeextracond} for the case $n=0$, we have
\begin{align*}
	\alpha_{n} =
	\begin{cases}
		\dfrac{1+q^{\frac{a}{2}}g_1}{1-q^{\frac{1}{2}}} & \text{if $n=0$}\\	
		q^{\frac{a+1}{2}n^2+n}\left( \dfrac{\bigl(g_n + q^{n(a-1)+ \frac{a}{2}}g_{n+1}\bigr)}{1-q^{n+\frac{1}{2}}} - \dfrac{q^{-\frac{1}{2}}\bigl(g_n + q^{-n(a-1) + \frac{a-2}{2}}g_{n-1}\bigr)}{1-q^{n - \frac{1}{2}}} \right) & \text{if $n>0$},
	\end{cases}
\end{align*}
and
\begin{align*}
	\beta_n =\beta_{n_{a+2}}
	=
	\frac{1}{ \bigl(-q^{\frac{1}{2}}\bigr)_{n_{a+2}}}
	\sum_{n_{a+1},\dots,n_1\geq 0}
	\dfrac{\bigl(1-q^{n_{a+1}+1}\bigr)q^{\frac{1}{2}(n_{a+1}^2+n_{a}^2+\cdots+n_1^2)}\bigl(-q^{\frac{1}{2}}\bigr)_{n_1}}{(q)_{n_{a+2}-n_{a+1}}(q)_{n_{a+1}-n_{a}+1}
		(q)_{n_{a}-n_{a-1}}
		\cdots
		(q)_{n_{2}-n_{1}}} \beta_{n_1}.
\end{align*}

{\it Step 5\emph{:}}
If $a\neq m$, we use Lemma \ref{Baileylemma} $m-a-1$ times with $x\mapsto 1$, \smash{$b\mapsto -q^{\frac{1}{2}}$} and $c\rightarrow \infty$. Converting the $\beta_n$ to $q$-binomial notation then gives the required Bailey pair relative to $x=1$.

If $a=m$, we use Lemma \ref{Baileylemma} with $x\mapsto 1$, \smash{$b\mapsto -q^{\frac{1}{2}}$}, and $c\rightarrow 0$
to obtain a new Bailey pair relative to $x=1$. This gives the correct $\alpha_n$ in \eqref{eqn:dilationpairalpha}, while
\begin{gather*}
	\beta_n =
	\dfrac{q^{-\frac{n^2}{2}}}{\bigl(-q^{\frac{1}{2}}\bigr)_{n}}
	\sum_{n_{m+2}=0}^n (-1)^{n-n_{m+2}}q^{\binom{n-n_{m+2}}{2}}\dfrac{\bigl(-q^{\frac{1}{2}}\bigr)_{n_{m+2}}}{(q)_{n-n_{m+2}}}\notag\\ \hphantom{\beta_n =}{}
	\times \frac{1}{ \bigl(-q^{\frac{1}{2}}\bigr)_{n_{m+2}}}
	\sum_{n_{m+1},\dots,n_1\geq 0}
	\dfrac{\bigl(1-q^{n_{m+1}+1}\bigr)q^{\frac{1}{2}(n_{m+1}^2+n_{m}^2+\cdots+n_1^2)}\bigl(-q^{\frac{1}{2}}\bigr)_{n_1}}{(q)_{n_{m+2}-n_{m+1}}(q)_{n_{m+1}-n_{m}+1}
		(q)_{n_{m}-n_{m-1}}
		\cdots
		(q)_{n_{2}-n_{1}}} \beta_{n_1}.
\end{gather*}
We now sum over $n_{m+2}$ first, using the $q$-binomial theorem \cite[equation~(3.3.6)]{AndBook}:
\begin{align*}
	\sum_{n_{m+2}=0}^n (-1)^{n-n_{m+2}}q^{\binom{n-n_{m+2}}{2}}
	\dfrac{1}{(q)_{n-n_{m+2}}(q)_{n_{m+2}-n_{m+1}}}=
	\delta_{n,n_{m+1}}.
\end{align*}
Thus, the $\beta_n$ simplifies to
\begin{align*}
	\beta_n &{}= \beta_{n_{m+1}}=
	\dfrac{1}{\bigl(-q^{\frac{1}{2}}\bigr)_{n_{m+1}}}
	\sum_{n_{m},\dots,n_1\geq 0}
	\dfrac{\bigl(1-q^{n_{m+1}+1}\bigr)q^{\frac{1}{2}(n_{m}^2+\cdots+n_1^2)}\bigl(-q^{\frac{1}{2}}\bigr)_{n_1}}{(q)_{n_{m+1}-n_{m}+1}
		(q)_{n_{m}-n_{m-1}}
		\cdots
		(q)_{n_{2}-n_{1}}} \beta_{n_1},
\end{align*}
and converting to $q$-binomial coefficients gives the required \eqref{eqn:dilationpairbeta}.

{\it Step 6\emph{:}}
If $a=0$, we skip Step 1 and use Steps 2--5. We obtain the Bailey pair with $\alpha_n$ given by \eqref{eqn:dilationpairalpha} for $a=0$ and
\begin{align*}
	\beta_n &{}=\beta_{n_{m+1}}
	=
	\frac{1}{ \bigl(-q^{\frac{1}{2}}\bigr)_{n_{m+1}}}
	\sum_{n_{m},\dots,n_1\geq 0}
	\dfrac{q^{\frac{1}{2}(n_{m}^2+\cdots+n_1^2)}\bigl(1-q^{n_1+1}\bigr)\bigl(-q^{\frac{1}{2}}\bigr)_{n_1+1}}
	{(q)_{n_{m+1}-n_{m}}
		\cdots
		(q)_{n_{2}-n_{1}}} \beta_{n_1+1}.
\end{align*}
Converting to $q$-binomial notation completes the proof.
\end{proof}

Armed with the Bailey pair in Lemma \ref{lemma:dilationpair}, we now prove Theorem \ref{thm:dilation}.

\begin{proof}[Proof of Theorem \ref{thm:dilation}]
Subtracting the case $a-1$ from the case $a$ of Lemma \ref{lemma:dilationpair}, we obtain a~new Bailey pair $(\alpha_n,\beta_n)$ relative to $1$. For $n=0$, we have
\begin{align*}
	\alpha_0 = \frac{1+q^{\frac{a}{2}}g_1}{1-q^{\frac{1}{2}}} - \frac{1+q^{\frac{a-1}{2}}g_1}{1-q^{\frac{1}{2}}} = -q^{\frac{a-1}{2}}g_1,
\end{align*}
and for $n > 0$ we have
\begin{align*}
	\alpha_n &{}= q^{\frac{m}{2}n^2+n}\left( \dfrac{\bigl(g_n + q^{n(a-1)+ \frac{a}{2}}g_{n+1}\bigr)}{1-q^{n+\frac{1}{2}}} - \dfrac{q^{-\frac{1}{2}}\bigl(g_n + q^{-n(a-1) + \frac{a-2}{2}}g_{n-1}\bigr)}{1-q^{n - \frac{1}{2}}} \right) \\
	&\quad{}- q^{\frac{m}{2}n^2+n}\left( \dfrac{\bigl(g_n + q^{n(a-2)+ \frac{a-1}{2}}g_{n+1}\bigr)}{1-q^{n+\frac{1}{2}}} - \dfrac{q^{-\frac{1}{2}}\bigl(g_n + q^{-n(a-2) + \frac{a-3}{2}}g_{n-1}\bigr)}{1-q^{n - \frac{1}{2}}} \right) \\
	&{}= -q^{\frac{m}{2}n^2 +n}\left(q^{n(a-2) + \frac{a-1}{2}}g_{n+1} + q^{-n(a-1) +\frac{a-3}{2}}g_{n-1}\right).
\end{align*}
Now using \eqref{eqn:dilationpairbeta} and \eqref{Baileylimit} along with the fact that
\begin{equation*}
	\lim_{n_{m+1} \to \infty} \begin{bmatrix} n_{m+1} + \delta_{a,m} \\ n_m \end{bmatrix} = \frac{1}{(q)_{n_m}},
\end{equation*}
we obtain
\begin{align*}
	\text{l.h.s. of \eqref{eqn:dilationid}} = \frac{\bigl(-q^{\frac{1}{2}}\bigr)_{\infty}}{(q)_{\infty}}\biggl(-q^{\frac{a-1}{2}}g_1 - \sum_{n \geq 1}q^{\frac{m}{2}n^2+n}\bigl(q^{n(a-2) + \frac{a-1}{2}}g_{n+1} + q^{-n(a-1) +\frac{a-3}{2}}g_{n-1}\bigr)\biggr).
\end{align*}
If $g_n$ satisfies \eqref{eqn:gcond}, then we compute
\begin{align*}
	\text{l.h.s. of \eqref{eqn:dilationid}}
	&{}= -q^{\frac{a-1}{2}}\frac{\bigl(-q^{\frac{1}{2}}\bigr)_{\infty}}{(q)_{\infty}} \biggl(\sum_{n \geq 0} q^{\frac{m}{2}n^2 +an -n}g_{n+1} + \sum_{n \geq 1} q^{\frac{m}{2}n^2 -an + 2n - 1}g_{n-1}\biggr) \\
	&{}= -q^{\frac{a-1}{2}}\frac{\bigl(-q^{\frac{1}{2}}\bigr)_{\infty}}{(q)_{\infty}} \biggl(\sum_{n \geq 0} q^{\frac{m}{2}n^2 +an -n}g_{n+1} + \sum_{n \leq -1} q^{\frac{m}{2}n^2 +an - 2n - 1}g_{-n-1} \biggr) \\
	&{}= -q^{\frac{a-1}{2}} \frac{\bigl(-q^{\frac{1}{2}}\bigr)_{\infty}}{(q)_{\infty}} \sum_{n \in \mathbb{Z}} q^{\frac{m}{2}n^2+na-n}g_{n+1},
\end{align*}
as desired.
\end{proof}

We now give several applications of Theorem \ref{thm:dilation}, beginning with Theorem \ref{thm:dilationex1}.

\begin{proof}[Proof of Theorem \ref{thm:dilationex1}]
We use the Bailey pair relative to $1$ \cite[p.~468]{Sl1},
\begin{align*}
	\alpha_n=
	\begin{cases}
		1 & \text{if $n=0$},\\
		(-1)^nq^{n^2-\frac{1}{2}n}(1+q^n) & \text{if $n>0$},
	\end{cases}
\qquad
\text{and}
\qquad
	\beta_n=\dfrac{1}{\bigl(-q^{\frac{1}{2}}\bigr)_n(q)_n},
\end{align*}
corresponding to $g_n = (-1)^nq^{n^2-\frac{n}{2}}$. Using this in Theorem \ref{thm:dilation} with $q=q^2$,
we obtain
\begin{align*}
	\mathscr{S}_{m,a} - \mathscr{S}_{m,a-1} &{}= \frac{\bigl(-q;q^2\bigr)_{\infty}}{\bigl(q^2;q^2\bigr)_{\infty}} q^{a}\sum_{n\in \mathbb{Z}}(-1)^nq^{(m+2)n^2-2na-n} \\
	&{}=q^{a}\frac{\bigl(-q;q^2\bigr)_{\infty}}{\bigl(q^2;q^2\bigr)_{\infty}}\cdot\bigl(q^{m+1-2a},q^{m+3+2a},q^{2m+4};q^{2m+4}\bigr)_\infty,
\end{align*}
by the triple product identity \eqref{eqn:jtp}.
\end{proof}

Next, for $m \geq 1$ and $0 \leq a \leq m$ define
\begin{equation*}
\mathscr{R}_{m,a} = \sum_{n_m, \dots, n_1 \geq 0} \frac{q^{n_m^2 + \cdots + n_1^2}\bigl(-q;q^2\bigr)_{n_1 + \delta_{a,0}}}{\bigl(q^2;q^2\bigr)_{n_m}} \prod_{i=1}^{m-1} \begin{bmatrix} n_{i+1} + \delta_{a,i} \\ n_i \end{bmatrix}_{q^2}.
\end{equation*}
\begin{Theorem}
For $1 \leq a \leq m$, we have
\begin{equation*}
	\mathscr{R}_{m,a} - \mathscr{R}_{m,a-1} = q^{a+1}\frac{\bigl(-q;q^2\bigr)_{\infty}}{\bigl(q^2;q^2\bigr)_{\infty}}\bigl(q^{m+6+2a},q^{m-2a},q^{2m+6};q^{2m+6}\bigr)_{\infty}.
\end{equation*}
\end{Theorem}

\begin{proof}
We use the Bailey pair relative to $1$ \cite[p.~468, B(1)]{Sl1},
\begin{align}
	\alpha_n=
	\begin{cases}
		1 & \text{if $n=0$},\\
		(-1)^nq^{n(3n-1)/2}(1+q^n) & \text{if $n>0$},
	\end{cases}
	\label{eqn:alphaR}
\end{align}
and
\begin{align}
	\beta_n=\dfrac{1}{(q)_n}.
	\label{eqn:betaR}
\end{align}
Applying Theorem \ref{thm:dilation} and then letting $q \mapsto q^2$ gives
\begin{align*}
	\mathscr{R}_{m,a} - \mathscr{R}_{m,a-1} &{}= -q^{a-1}\frac{\bigl(-q;q^2\bigr)_{\infty}}{\bigl(q^2;q^2\bigr)_{\infty}}\sum_{n \in \mathbb{Z}} q^{mn^2+2an-2n}(-1)^{n+1}q^{(n+1)(3n+2)} \\
	&{}= q^{a+1}\frac{\bigl(-q;q^2\bigr)_{\infty}}{\bigl(q^2;q^2\bigr)_{\infty}}\sum_{n \in \mathbb{Z}} q^{(m+3)n^2+(2a+3)n}(-1)^{n},
\end{align*}
and the result follows from the triple product identity \eqref{eqn:jtp}.
\end{proof}
Analogous to \eqref{eqn:Smm}, we have that $\mathscr{R}_{m,m}$ are products for $m\geq 1$.
\begin{Theorem}
For $m\geq 1$, we have
\begin{align*}	
	\mathscr{R}_{m,m}=\frac{\bigl(-q;q^2\bigr)_{\infty}}{\bigl(q^2;q^2\bigr)_{\infty}}\bigl(q^{m+4},q^{m+2},q^{2m+6};q^{2m+6}\bigr)_{\infty}.
\end{align*}
\end{Theorem}
\begin{proof}
We start with the Bailey pair in \eqref{eqn:alphaR} and \eqref{eqn:betaR}, and then consider the Bailey pair given by \eqref{eqn:alphathm1.6step1} and \eqref{eqn:betathm1.6step1} with $a=m$. We then apply \eqref{Baileylimit}. Observing that
\begin{align*}
	\sum_{n\geq 0} \alpha_n &{}=1+ \sum_{n\geq 1} q^{\frac{m}{2}n^2}(1+q^n)(-1)^nq^{n(3n-1)/2}
	 =\sum_{n\in\ZZ}(-1)^n q^{\frac{m+3}{2}n^2-\frac{n}{2}} \\
	&{}=\bigl(q^{\frac{m}{2} +2},q^{\frac{m}{2} +1},q^{m+3};q^{m+3}\bigr)_\infty
\end{align*}
by \eqref{eqn:jtp}, the result follows after dilating by $q\mapsto q^2$.
\end{proof}

For our final application, for $m \geq 1$ and $0 \leq a \leq m$ define
\begin{equation*}
\mathscr{T}_{m,a} = \sum_{n_m, \dots, n_1 \geq 0} \frac{q^{n_m^2 + \cdots + n_1^2}\bigl(-1;q^2\bigr)_{n_1 + \delta_{a,0}}}{\bigl(q^2;q^2\bigr)_{n_m}\bigl(q,q^2\bigr)_{n_1 + \delta_{a,0}}} \prod_{i=1}^{m-1} \begin{bmatrix} n_{i+1} + \delta_{a,i} \\ n_i \end{bmatrix}_{q^2}.
\end{equation*}

\begin{Theorem}
For $1 \leq a \leq m$, we have
\begin{equation*}
	\mathscr{T}_{m,a} - \mathscr{T}_{m,a-1} = -q^{a-1}\frac{\bigl(-q;q^2\bigr)_{\infty}}{\bigl(q^2;q^2\bigr)_{\infty}}\bigl(-q^{m+2a},-q^{m+2-2a},q^{2m+2};q^{2m+2}\bigr)_{\infty}.
\end{equation*}
\end{Theorem}

\begin{proof}
In Theorem \ref{thm:dilation}, we argue as usual using the Bailey pair relative to $1$ \cite[p.~468, corrected]{Sl1},
\begin{align}
	\alpha_n=
	\begin{cases}
		1 & \text{if $n=0$},\\
		q^{\binom{n}{2}}(1+q^n) & \text{if $n>0$},
	\end{cases}
	\label{eqn:alphaT}
\end{align}
and
\begin{align}
	\beta_n=\dfrac{(-1)_n}{\bigl(-q^{\frac{1}{2}},q^{\frac{1}{2}},q\bigr)_n}.
	\label{eqn:betaT}
\end{align}
This concludes the proof.
\end{proof}

Again, the $\mathscr{T}_{m,m}$ are products for $m\geq 1$.
\begin{Theorem}
For $m\geq 1$, we have
\begin{align*}	\mathscr{T}_{m,m}=\frac{\bigl(-q;q^2\bigr)_{\infty}}{\bigl(q^2;q^2\bigr)_{\infty}}\bigl(-q^{m},-q^{m+2},q^{2m+2};q^{2m+2}\bigr)_{\infty}.
\end{align*}
\end{Theorem}
\begin{proof}
We start with the Bailey pair in \eqref{eqn:alphaT} and \eqref{eqn:betaT}, and then consider the Bailey pair given by \eqref{eqn:alphathm1.6step1} and \eqref{eqn:betathm1.6step1} with $a=m$. We then apply \eqref{Baileylimit}. Observing that
\begin{align*}
	\sum_{n\geq 0} \alpha_n &{}=1+ \sum_{n\geq 1} q^{\frac{m}{2}n^2}(1+q^n)q^{\frac{n(n-1)}{2}}
	 =\sum_{n\in\ZZ}q^{\frac{m+1}{2}n^2-\frac{n}{2}}
	 =\bigl(-q^{\frac{m}{2}} ,-q^{\frac{m}{2}+1},q^{m+1};q^{m+1}\bigr)_\infty
\end{align*}
by \eqref{eqn:jtp} gives the result after dilating by $q\mapsto q^2$.
\end{proof}

\section{Concluding remarks}
In this paper, we have shown how the families of $q$-series identities involving products of $q$-binomial coefficients like the one in \eqref{eqn:shiftedbinom} fit naturally into the theory of Bailey pairs. As noted in the introduction, this is not a surprise, given that several similar identities have already appeared in the literature in relation to Bailey pairs. For example, Hikami's variant of the Andrews--Gordon identities \cite{Hi1},
\begin{equation*}
\sum_{n_{m-1}, \dots, n_1 \geq 0} \frac{q^{n_1^2 + \cdots+ n_{m-1}^2 + n_{a+1} + \cdots + n_{m-1}}}{(q)_{n_{m-1}}} \prod_{i=1}^{m-2} \begin{bmatrix} n_{i+1} + \delta_{a,i} \\ n_i \end{bmatrix} = \frac{\bigl(q^{a+1},q^{2m-a},q^{2m+1};q^{2m+1}\bigr)_{\infty}}{(q)_{\infty}},
\end{equation*}
valid for $m \geq 2$ and $0 \leq a \leq m-1$, can be proved using the theory of Bailey pairs, as can several other similar families of infinite product and false theta identities, \cite{Hi1,Lo1,Lo-Sa1}.

Identities involving products of $q$-binomial coefficients similar to \eqref{eqn:shiftedbinom} also arose in connection with Schur's indices of certain $4d$ $N=2$ Argyres--Douglas theories in \cite{KMR}. For $t\geq 1$, $1\leq s\leq t+1$ define (with the convention that $n_{t+1}=0$):
\begin{align*}
\mathscr{D}_{t,s}=
\sum_{n_1,\dots, n_t\geq 0}
q^{n_s}\prod_{r=1}^{t}\frac{q^{n_r n_{r+1} + n_r} }{(q)_{n_r}^2}.
\end{align*}
Then, for $k\geq 1$, $0\leq i\leq k$, it was proved in \cite[Theorems~6.1 and~7.1]{KMR} that
\begin{align}
&(-1)^{k-i}\mathscr{D}_{2k,k+i+1}
+2\sum_{j=i+1}^{k}(-1)^{k-j}\mathscr{D}_{2k,k+j+1}
= \dfrac{\bigl(q^{k-i+1},q^{k+i+2},q^{2k+3}; q^{2k+3}\bigr)_{\infty}}{(q)_\infty^{2k+1}},\label{kmr:ag}\\
&(-1)^{k-i}\mathscr{D}_{2k-1,k+i}
+2\sum_{j=i+1}^{k}(-1)^{k-j}\mathscr{D}_{2k-1,k+j}
= \dfrac{1}{(q)_{\infty}^{2k}}\sum_{n\in\mathbb{Z}}\mathrm{sgn}(n)q^{(k+1)n^2+in}
\label{kmr:false}.
\end{align}
The first step in the proof of these identities involves using quantum dilogarithm to deduce
\begin{gather}
\mathscr{D}_{t,s}
=\dfrac{1}{(q)_{\infty}^t}
\sum_{m_1,\dots,m_{t-1}\geq 0}
(-1)^{\sum_{j=2}^{t-1}m_j}q^{\binom{m_1+1}{2}+ \sum_{j=1}^{t-1}\binom{m_j+1}{2}}
\dfrac{(q)_{m_{t-1} + \delta_{t,s}}}{(q)_{m_1}^2}
\prod_{i=1}^{t-2}
\begin{bmatrix}
	m_{i} + \delta_{i,s-1} \\
	m_{i+1}
\end{bmatrix}\!\!\!\!
\label{kmr:ag'}
\end{gather}
for $t\geq 2$, $2\leq s\leq t+1$. Here we have taken the liberty to rewrite the expressions in \cite[Propo\-si\-tion~4.2]{KMR} to more transparently exhibit the $q$-binomial coefficients of the shape \eqref{eqn:shiftedbinom}.
Crucially, the next step in the proof of \eqref{kmr:ag} and \eqref{kmr:false} uses the theory of Bailey pairs to go from~\eqref{kmr:ag'} to the theta and false theta counterparts, as appropriate \cite[Section~5]{KMR}.
Note the slight change in the order of the variables:
The summation variables in Theorem \ref{CLSXYthm} are ordered subject to $n_{i+1}+\delta_{a,i}\geq n_i$, however in \eqref{kmr:ag'}, they are ordered $m_i+\delta_{i,s-1}\geq m_{i+1}$. We omit the discussion of two further families of identities involving $w$-deformed $\mathscr{D}$ functions proved in \cite{KMR}.

\subsection*{Acknowledgements}

We thank A.J.\ Yee for her talk at the International Conference on Modular Forms and $q$-Series at University of Cologne in March 2024. We also thank her for facilitating this collaboration. SK gratefully acknowledges the support of Simons Foundations' Travel Support for Mathematicians \#636937 and MPS-TSM-00007726.


\pdfbookmark[1]{References}{ref}
\LastPageEnding

\end{document}